\documentclass[12pt,reqno]{amsart}
\usepackage{amsmath}
\usepackage{amssymb}
\usepackage{amstext}
\usepackage{a4wide}
\usepackage{graphicx}
\allowdisplaybreaks \numberwithin{equation}{section}
\usepackage{color}
\usepackage{cases}

\numberwithin{equation}{section}

\newtheorem{theorem}{Theorem}[section]
\newtheorem{proposition}[theorem]{Proposition}

\newtheorem{lemma}[theorem]{Lemma}

\theoremstyle{definition}

\newtheorem{definition}[theorem]{Definition}

\theoremstyle{remark}
\newtheorem{remark}[theorem]{Remark}

\begin{document}

\title
{Desingularization of Vortex Rings in 3 dimensional Euler Flows: with swirl}

 \author{Daomin Cao, Jie Wan,  Weicheng Zhan}

\address{Institute of Applied Mathematics, Chinese Academy of Sciences, Beijing 100190, and University of Chinese Academy of Sciences, Beijing 100049,  P.R. China}
\email{dmcao@amt.ac.cn}
\address{Institute of Applied Mathematics, Chinese Academy of Sciences, Beijing 100190, and University of Chinese Academy of Sciences, Beijing 100049,  P.R. China}
\email{wanjie15@mails.ucas.edu.cn}
\address{Institute of Applied Mathematics, Chinese Academy of Sciences, Beijing 100190, and University of Chinese Academy of Sciences, Beijing 100049,  P.R. China}
\email{zhanweicheng16@mails.ucas.ac.cn}


\begin{abstract}
We study desingularization of steady vortex rings in three-dimensional axisymmetric incompressible Euler fluids with swirl. Using the variational method, we construct a two-parameter family of steady vortex rings, which constitute a desingularization of the classical circular vortex filament, in several kinds of domains. The precise localization of the asymptotic singular vortex filament is shown to depend on the circulation and the velocity at far fields of the vortex ring and the geometry of the domains. We also discuss other qualitative and asymptotic properties of these vortices.
\end{abstract}

\maketitle

\section{Introduction}

In this paper, we study the three-dimensional axisymmetric Euler flows with swirl. The motion of general incompressible steady Euler fluid in $\mathbb{R}^3$ is as follows

\begin{equation}\label{1-1}
~(\mathbf{v}\cdot\nabla)\mathbf{v}=-\nabla P,
\end{equation}
\begin{equation}\label{1-2}
 \nabla\cdot\mathbf{v}=0,
\end{equation}
where $\mathbf{v}=[v_1,v_2,v_3]$ is the velocity field and $P$ is the scalar pressure. Let $\{\mathbf{e}_r, \mathbf{e}_\theta, \mathbf{e}_z\}$ be the usual cylindrical coordinate frame. Then if the velocity field $\mathbf{v}$ is axisymmetric, i.e., $\mathbf{v}$ does not depend on the $\theta$ coordinate, it can be expressed in the following way
\begin{equation*}
  \mathbf{v}=v^r(r,z)\mathbf{e}_r+v^\theta(r,z)\mathbf{e}_\theta+v^z(r,z)\mathbf{e}_z.
\end{equation*}
 The component $v^\theta$ in the $\mathbf{e}_\theta$ direction is called the swirl velocity. Let $\pmb{\omega}:=\nabla\times\mathbf{v}$ be the corresponding vorticity field. We shall refer to the axisymmetric flows  as ``vortex rings with swirl" if there is a toroidal region inside of which the $\theta$-component of velocity $v^\theta$ and of vorticity $\omega^\theta$ are nonzero, and outside of which the flow is irrotational.

The study of vortex rings with swirl, which arises from aerodynamics and hydrodynamics, can be traced back to Hicks in 1899 (see \cite{H}). Compared with vortex rings without swirl\,($v^\theta\equiv 0$)\,which have been studied extensively and deeply, the dynamics of vortex rings with swirl seems to have received relatively less attention. However, most vortical structures in transitional and turbulent flows seem to have a swirl component (refer to \cite{Virk94}). Vortex rings with swirl are of interest for a number of reasons. The interactions of the flow velocities $v^r$ and $v^z$ in the axial planes with the swirl velocity $v^\theta$ usually lead to some interesting phenomena. It is often viewed as an important model for understanding turbulence physics and controlling turbulence phenomena (see \cite{Na14}). This flow may also enhance our understanding of axisymmetric vortex breakdown (see, e.g.,  \cite{BS, Lu2001, WR}). Moreover, in contrast with the non-swirl\,($v^\theta\equiv 0$)\,Euler flows, flows with swirl may have nonzero helicity, which directly related to the topological invariance of the corresponding vorticity field (see \cite{Ban16, M}).

The only explicit solution for vortices with swirl was first found by Hicks \cite{H} and rediscovered by Moffatt \cite{M}. In \cite{M2}, Moffatt adopted magnetic relaxation methods to obtain a wide family of vortex rings with swirl. In the study of the existence of vortex rings with swirl, the problem is often translated into the following nonlinear elliptic equation (called the Bragg-Hawthorne equation \cite{BH})
\begin{equation}\label{1-3}
\mathcal{L}\psi:=-\frac{1}{r}\frac{\partial}{\partial r}\Big(\frac{1}{r}\frac{\partial\psi}{\partial r}\Big)-\frac{1}{r^2}\frac{\partial^2\psi}{\partial z^2}=-B'(\psi)+\frac{1}{r^2}H(\psi)H'(\psi)
\end{equation}
with suitable boundary conditions. Here the Bernoulli function (or dynamic pressure) $B:=\frac{1}{2}|\mathbf{v}|^2+P$ and $H:=rv^{\theta}$ are two prescribed functions of $\psi$. This family of equations is also sometimes referred to as the Long-Squire equation (see \cite{Long53, Sq56}). It is usually known as the Grad-Shafranov equation in plasma physics (see \cite{An98}). Once the Stokes stream function $\psi$ is obtained, one can easily construct the corresponding solutions of the original or primitive variables $(\mathbf{v}, P)$ and vice versa. In \cite{EM}, Elcrat and Miller developed an iterative procedure to solve the nonlinear equation $\eqref{1-3}$. Using this monotone method one may construct a family of vortex rings with swirl based on arbitrary solutions of vortex rings without swirl. In addition, the authors also investigated vortex breakdown phenomenon in a bounded channel. Recently, Frewer et al. studied the Bragg-Hawthorne equation via a Lie point symmetry analysis \cite{Fre07}. For numerical and experimental studies of vortex rings with swirl, see \cite{Ch10, EM2, ET, Ga15, Na14, Virk94} and references therein. \cite{Lim95, MGT, Sa92} are some good historical reviews of the achievements in experimental, analytical, and numerical studies of vortex rings.

In this paper, we are concerned with the nonlinear desingularization of vortex rings with swirl. Limited work has been done in this aspect. In \cite{Tur89}, Turkington constructed a two-parameter family of desingularized steady solutions in a bounded domain and in the whole space respectively. His key idea was based on a reformulation of the steady equations in terms of the modified azimuthal vorticity $\zeta:=\omega^\theta/r$ alone.  In \cite{Ta2}, Tadie considered Euler flows outside infinite cylinders and investigated the asymptotic behaviour by letting the flux diverge. In the present work, we follow Turkington's idea to construct a family of steady vortex rings, which constitute a desingularization of the classical circular vortex filament, in several kinds of domains. However, we would like to point out that the method we adopt here is different in some important aspects from the one Turkington proposed in \cite{Tur89}. Notice that in order to obtain a family of desingularized solutions in the whole space, Turkington imposed some additional constrains which cause the velocities of the flows at far fields became Lagrange multipliers and hence were undetermined (see~Theorem 2 of \cite{Tur89}). However, the velocities at infinity of our solutions are determined. In the case of exterior domains, one cannot obtain the compactness by using embedding of sets
of symmetric functions as in \cite{Tur89}. Our method can effectively circumvent this difficulty (see Section 4 below). In addition, the precise localization of the asymptotic singular vortex filament is shown to depend on the circulation and the velocity at far fields of the vortex ring, and the geometry of the domains. Our work here may be thought of as an extension of Turkington's work.

 Inspired by the variational principle proposed by Turkington (see (VP2) of \cite{Tur89}), we can transform the problem into a constrained maximization problem in $\zeta$, the solutions of which define the desired steady vortex rings with swirl. For simplicity, we will restrict our discussion to the ``classical'' case of $B$ and $H$ linearly dependent upon the stream function $\psi$. It is the simplest choice of $B$ and $H$, so that one can analyze the model without many difficulties whether in theory or in numerical computation (For numerical investigations, see \cite{EM2,ET}). The general case can be treated similarly; see \cite{Tur89} for more discussion. Hence we now turn to the problem of energy maximization on proper admissible class. To maximize the objective (energy) functional, those vortices constructed have to be concentrated. This method can also work on the problems of three-dimensional axisymmetric vortex rings without swirl (see, e.g., \cite{Be,B1,CWZ,D,Dou2,FT}). This work can be viewed as a continuation of the recent work \cite{CWZ} about vortex rings without swirl.

The paper is organized as follows. In section 2, we briefly deduce the system of vortex rings with swirl and then state the main results. In section 3, we study vortex rings in an infinite pipe. In section 4, we consider vortex rings outside a ball. In section 5, we briefly investigate vortex rings in the whole space.

\section{Main results}
Throughout the sequel we shall use the following notations: Let $\Pi=\{(r,z)~|~r>0, z\in\mathbb{R}\}$ denote a meridional half-plane ($\theta$=constant); $B_\delta(y)$ denotes an open ball in $\mathbb{R}^2$ of center $y$ and radius $\delta>0$; Let $\chi_{_A}$ denote the characteristic function of $A\subseteq\Pi$; $f_+$ denotes the positive parts of $f$; Lebesgue measure on $\mathbb{R}^N$ is denoted $\textit{m}_N$, and is to be understood as the measure defining any $L^p$ space and $W^{1,p}$ space, except when stated otherwise; $\nu$ denotes the measure on $\Pi$ having density $r$ with respect to $\textit{m}_2$ and $|\cdot|$ denotes the $\nu$-measure.

In the following, we first deduce the system of vortex rings with swirl in brief. For more details, we refer to \cite{Tur89}.

By system \eqref{1-1}-\eqref{1-2}, the axisymmetric 3D Euler flows can be characterized as follows
\begin{numcases}
{}
\label{1-4} \mathbf{v}\cdot\nabla v^z=-\partial_z P,  &\\
\label{1-5} \mathbf{v}\cdot\nabla v^r-\frac{1}{r}(v^{\theta})^2=-\partial_r P,  &\\
\label{1-6} \mathbf{v}\cdot\nabla v^{\theta}+\frac{1}{r}v^r v^{\theta}=0,  &\\
\label{1-7} \partial_z v^z+\frac{1}{r}\partial_r(rv^r)=0,  &
\end{numcases}
where $\mathbf{v}\cdot\nabla=v^r\partial_r+v^z\partial_z$.
The vorticity field $\pmb{\omega}=(\omega^r,\omega^{\theta},\omega^z)$ is given by
\begin{equation}\label{1-8}
   \omega^r=-\partial_z v^{\theta}, \ \  \omega^{\theta}=\partial_z v^r-\partial_r v^z,\ \  \omega^z=\frac{1}{r}\partial_r (r v^{\theta}).
\end{equation}
By the continuity equation \eqref{1-7}, we can find a Stokes stream function $\psi=\psi(r,z)$ such that
\begin{equation}\label{1-9}
   v^r=-\frac{1}{r}{\partial_z\psi},\ \  \  v^z=\frac{1}{r}{\partial_r\psi}.
\end{equation}
Let $\zeta:={\omega^{\theta}}/{r}$ and $\xi:=r v^{\theta}$. By \eqref{1-8} and \eqref{1-9}, we get $\zeta= \mathcal{L}\psi$.
It follows from \eqref{1-4} and \eqref{1-5} that
\begin{equation*}
\mathbf{v}\cdot\nabla\zeta-2r^{-4}\xi\partial_z\xi=0.
\end{equation*}
And \eqref{1-6} is equivalent to
\begin{equation*}
\mathbf{v}\cdot\nabla\xi=0.
\end{equation*}
Therefore the triple $(\psi, \xi, \zeta)$ satisfies the following system
\begin{numcases}
{}
\label{1-16} \partial(\psi,\xi)=0,  &\\
\label{1-17} \partial(\zeta,\psi)+\partial(\xi,\frac{\xi}{r^2})=0,  &
\end{numcases}
where $\partial(\cdot,\cdot)$ denotes the determinant of the gradient matrix, namely, $$\partial(f,g)=\partial_rf\partial_zg-\partial_zf\partial_rg.$$

Now we reduce this system to a single equation of $\zeta$. Formally, if there exists some function $b(t)$ satisfying
\begin{equation}\label{1-18}
\psi=b(\xi),
\end{equation}
then the pair $(\psi, \xi)$ satisfies \eqref{1-16}.
Substitution of  \eqref{1-18} into \eqref{1-17} yields to
\begin{equation*}
0=\partial (\zeta, b(\xi))+\partial (\xi, \frac{\xi}{r^2})=\partial (\zeta b'(\xi)-\frac{\xi}{r^2}, \xi).
\end{equation*}
Thus, if there exists some function $a(t)$ such that
\begin{equation}\label{1-19}
\zeta b'(\xi)-\frac{\xi}{r^2}=-a'(\xi),
\end{equation}
then the triple $(\psi,\xi,\zeta)$ also satisfies \eqref{1-17}. Recall that $\mathcal{L}\psi=\zeta$, consequently $\psi$ is determined by $\zeta$ alone.
Taking \eqref{1-18} into \eqref{1-19}, we get a single equation of $\zeta$, namely,
\begin{equation}\label{1-20}
\zeta=-\frac{a'(b^{-1}(\psi))}{b'(b^{-1}(\psi))}+\frac{b^{-1}(\psi)}{r^2 b'(b^{-1}(\psi))}.
\end{equation}
Note that this model is consistent with the Brag-Hawthorne equation \eqref{1-3} with $B(\psi)={a(b^{-1}(\psi))}$ and $H(\psi)=b^{-1}(\psi)$.

As in \cite{Tur89}, we will restrict our discussion to the classical case when
\begin{equation*}
  a(t)=-\alpha t,\ \   b(t)=\beta t\ \ \text{for given positive numbers}\  \alpha\ \text{and}\ \beta.
\end{equation*}
The general case can be treated similarly; see \cite{Tur89} for more discussion. Since the solutions we considered here are discontinuous (see Theorems \ref{thm1} and \ref{thm2} below), it is natural to understand the system \eqref{1-16}-\eqref{1-17} in some weak sense.

Let $U\subseteq \mathbb{R}^3$ be an axisymmetric domain about the $z$ axis. Let $D=U\cap\Pi$.
\begin{definition}
  Let $(\psi,\xi,\zeta)\in W^{2,p}_{\text{loc}}(D)\times W^{1,p}_{\text{loc}}(D)\times L^p_{\text{loc}}(D)$ for some $p>2$. We say that $(\psi,\xi,\zeta)$ is a weak solution of the system \eqref{1-16}-\eqref{1-17} if for any $\varphi\in C_0^\infty(D)$,
\begin{equation}\label{1-22}
  \int_D \partial(\psi,\xi)\varphi drdz=0,
\end{equation}
and
\begin{equation}\label{1-23}
  \int_D \big[\zeta \partial(\psi,\varphi)-\frac{1}{2r^2}\partial(\xi^2,\varphi)\big]drdz=0.
\end{equation}
\end{definition}

In the present paper, we mainly consider the following three types of domains. We note that our method can also be applied to more general domains.
\begin{definition}
The set $D$ is admissible if it is one of the following three types
\begin{itemize}
  \item[(a)]  $\{(r,z)\in \Pi~|~0<r<d\}$\  for some $d>0$,
  \item[(b)]  $\{(r,z)\in \Pi ~|~ r^2+z^2>d^2\}$\ for some $d>0$,
  \item[(c)]  $\Pi$.
\end{itemize}
\end{definition}

 We define the inverse of $\mathcal{L}$ as follows. One can check that the operator $\mathcal{K}$ is well-defined; see, e.g., \cite{BB,B1}.
\begin{definition}
Let D be admissible. The Hilbert space $H(D)$ is the completion of $C_0^\infty(D)$ with the scalar products
\begin{equation*}
  \langle u,v\rangle_H=\int_D\frac{1}{r^2}\nabla u\cdot\nabla v d\nu.
\end{equation*}
We define inverses $\mathcal{K}$ for $\mathcal{L}$ in the weak solution sense,
\begin{equation}\label{2-1}
  \langle \mathcal{K}u,v\rangle_H=\int_D uv d\nu \ \ \ for\  all\  v\in H(D), \ \ when \ u \in  L^{10/7}(D,r^3drdz).
\end{equation}
\end{definition}

Let $K(r,z,r',z')$ be the Green's function of $\mathcal{L}$ in $D$, with respect to zero Dirichlet data and measure $rdrdz$. One can easily show that the operator $\mathcal{K}$ is an integral operator with kernel $K(r,z,r',z')$ for all cases considered in this paper. We shall use this Green's representation formula directly without further explanation.

As in \cite{DV}, for an axisymmetric set $A\subseteq \mathbb{R}^3$ we introduce the axisymmetric distance as follows
\begin{equation*}
  dist_{\mathcal{C}_r}(A)=\sup_{x\in A}\inf_{x'\in{\mathcal{C}_r}}|x-x'|,
\end{equation*}
where $\mathcal{C}_r:=\{x\in\mathbb{R}^3~|x_1^2+x_2^2=r^2,z=0\}$ for some $r>0$.

We define the circulation of a vortex as follows
\begin{equation*}
  \kappa(\zeta)= \int_D \zeta d\nu.
\end{equation*}

Having made these preparations, we now state the main results. Our first result is on the existence of desingularized vortex rings in infinite cylinders.
\begin{theorem}\label{thm1}
Let $U=\{(r,\theta,z)\in \mathbb{R}^3~|~ r<d\}$ for some $d>0$ and let $D=U\cap\Pi$. Then for every $W>1/(4\pi d)$, $\alpha\ge0$ and all sufficiently small $\beta>0$, there exists a weak solution $(\psi_{_\beta},\xi_{_\beta},\zeta_{_\beta})$ of system \eqref{1-16}-\eqref{1-17} with the following properties:
\begin{itemize}
\item[(i)]For any $p>1$, $0<\gamma<1$, $\psi_{_\beta}\in W^{2,p}_{\text{loc}}(D)\cap C^{1,\alpha}(\bar{D})$ and satisfies
\begin{equation*}
  \mathcal{L}\psi_{_\beta}=\zeta_{_\beta}\ \ \text{a.e.} \ \text{in} \ D.
\end{equation*}

\item[(ii)] $(\psi_{_\beta},\xi_{_\beta},\zeta_{_\beta})$ is of the form
\begin{equation*}
      \psi_{_\beta}=\mathcal{K}\zeta_{_\beta}-\frac{Wr^2}{2}\log{\frac{1}{\beta}}-\mu_{_\beta},\ \ \xi_{_\beta}=\frac{1}{\beta}(\psi_{_\beta})_+,\ \ \zeta_{_\beta}=\frac{1}{r^2\beta^2}{(\psi_{_\beta})}_{+}+\frac{\alpha}{\beta}\chi_{_{\{\psi_{_\beta}>0\}}},
\end{equation*}
for some $\mu_{_\beta}>0$ depending on $\beta$. Furthermore,
\begin{equation*}
  \kappa(\zeta_{_\beta})=1\ \  \text{and}\ \ \ \xi_{_\beta} \le \frac{R_0}{\beta},
\end{equation*}
for some positive constant $R_0$ independent of $\beta$.

 \item[(iii)]Let $\Omega_\beta:=supp(\xi_{_\beta})=supp(\zeta_{_\beta})$, there exists some constant $R_1>1$ independent of $\beta$ such that
 \begin{equation*}
   diam(\Omega_\beta)\le R_1\beta.
 \end{equation*}
 Moreover,
 \begin{equation*}
   \lim_{\lambda \to +\infty}dist_{\mathcal{C}_{r_{_*}}}(\Omega_\beta)=0,
\end{equation*}
where $r_*=\frac{1}{4\pi W}$.
Furthermore, as $\beta \to 0^+$,
 $$\mu_{_\beta}=\frac{3}{32}\frac{1}{\pi^2W}\log{\frac{1}{\beta}}+O(1).$$
\item[(iv)] Let $$\mathbf{v}_{_\beta}=\frac{1}{r}\Big(-\frac{\partial\psi_{_\beta}}{\partial z}\mathbf{e}_r+\frac{1}{\beta}(\psi_{_\beta})_+\mathbf{e}_\theta+\frac{\partial\psi_{_\beta}}{\partial r}\mathbf{e}_z\Big),$$ then
\begin{equation*}
\begin{split}
    & \mathbf{v_{_\beta}}\cdot\mathbf{n}=0 \  \ \text{on}\ \partial U,\\
    & \mathbf{v}_{_\beta}\to -W\log\frac{1}{\beta}~\mathbf{e}_z\ \  \text{at}\ \infty,
\end{split}
\end{equation*}
where $\mathbf{n}$ is the unit outward normal of $\partial U$. Moreover, as $r\to 0$,
\begin{equation*}
     \frac{1}{r}\frac{\partial\psi_{_\beta}}{\partial z}\to 0,\ \ \frac{1}{r\beta}(\psi_{_\beta})_+\to 0\ \text{and}~ \ \frac{1}{r}\frac{\partial\psi_{_\beta}}{\partial r}\  \text{approaches a finite limit}.
\end{equation*}
\item[(v)]Let the center of vorticity be
\begin{equation*}
  X_\beta=\int_D x\zeta_{_\beta}(x)d\textit{m}_2(x),
\end{equation*}
and define the rescaled version of $\zeta_{_\beta}$ to be
\begin{equation*}
  g_{_\beta}(x)={\beta^2}\zeta_{_\beta}(X_\beta+\beta x).
\end{equation*}
Then every accumulation points as $\beta \to 0^+$, in the weak topology of $L^2$, of $\{g_{_\beta}:\beta>0\}$ are radially nonincreasing functions.
\end{itemize}
\end{theorem}

\begin{remark}
Kelvin and Hicks showed that if the vortex ring with circulation $\kappa$ has radius $r_*$ and its cross-section $\varepsilon$ is small, then the vortex ring moves at the velocity (see,e.g.,\cite{Lamb}).
\begin{equation}\label{KH}
  \frac{\kappa}{4\pi r_*}\Big(\log \frac{8r_*}{\varepsilon}-\frac{1}{4}\Big).
\end{equation}
One can see that our result is consistent with the Kelvin-Hicks formula.
\end{remark}

Our second main result is devoted to the construction of desingularized vortex rings outside a ball.
\begin{theorem}\label{thm2}
Let $U=\{(r,\theta,z)\in \mathbb{R}^3~|~r^2+z^2>d^2\}$ for some $d>0$ and let $D=U\cap\Pi$. Then for every $W<1/(6\pi d)$, $\alpha\ge0$ and all sufficiently small $\beta>0$, there exists a weak solution $(\psi_{_\beta},\xi_{_\beta},\zeta_{_\beta})$ of system \eqref{1-16}-\eqref{1-17} with the following properties:
\begin{itemize}
\item[(i)]For any $p>1$, $0<\gamma<1$, $\psi_{_\beta}\in W^{2,p}_{\text{loc}}(D)\cap C_{\text{loc}}^{1,\gamma}(\bar{D})$ and satisfies
\begin{equation*}
  \mathcal{L}\psi_{_\beta}=\zeta_{_\beta}\ \ \text{a.e.} \ \text{in} \ D.
\end{equation*}

\item[(ii)] $(\psi_{_\beta},\xi_{_\beta},\zeta_{_\beta})$ is of the form
\begin{equation*}
\begin{split}
    \psi_{_\beta} & =\mathcal{K}\zeta_{_\beta}-\frac{Wr^2}{2}\log{\frac{1}{\beta}}+\frac{Wr^2d^3}{2(r^2+z^2)^{\frac{3}{2}}}\log\frac{1}{\beta}-\mu_{_\beta},\\
     \xi_{_\beta} &=\frac{1}{\beta}(\psi_{_\beta})_+,\ \ \zeta_{_\beta}  =\frac{1}{r^2\beta^2}{(\psi_{_\beta})}_{+}+\frac{\alpha}{\beta}\chi_{_{\{\psi_{_\beta}>0\}}},
\end{split}
\end{equation*}
for some $\mu_{_\beta}>0$ depending on $\beta$. Furthermore,
\begin{equation*}
  \kappa(\zeta_{_\beta})=1\ \  \text{and}\ \ \ \xi_{_\beta} \le \frac{R_0}{\beta},
\end{equation*}
for some positive constant $R_0$ independent of $\beta$.

\item[(iii)]Let $\Omega_\beta:=supp(\xi_{_\beta})=supp(\zeta_{_\beta})$, there exists some constant $R_1>1$ independent of $\beta$ such that
 \begin{equation*}
   diam(\Omega_\beta)\le R_1\beta.
 \end{equation*}
 Moreover,
 \begin{equation*}
   \lim_{\lambda \to +\infty}dist_{\mathcal{C}_{r_{_*}}}(\Omega_\beta)=0,
\end{equation*}
where $r_*\in [d,+\infty)$ satisfies $\Gamma_2(r_*)=\max_{t\in[d,+\infty)}\Gamma_2(t)$ and
\begin{equation*}
    \Gamma_2(t):=\frac{t}{2\pi}-Wt^2+\frac{Wd^3}{t},\ t\in(0,+\infty).
\end{equation*}
Furthermore, as $\beta \to 0^+$,
 $$ \mu_{_\beta} =(\frac{r_*}{2\pi}-\frac{Wr_*^2}{2}+\frac{Wd^3}{2r_*})\log{\frac{1}{\beta}}+O(1).$$
\item[(iv)]  Let $$\mathbf{v}_{_\beta}=\frac{1}{r}\Big(-\frac{\partial\psi_{_\beta}}{\partial z}\mathbf{e}_r+\frac{1}{\beta}(\psi_{_\beta})_+\mathbf{e}_\theta+\frac{\partial\psi_{_\beta}}{\partial r}\mathbf{e}_z\Big),$$ then
\begin{equation*}
\begin{split}
    & \mathbf{v_{_\beta}}\cdot\mathbf{n}=0 \  \ \text{on}\ \partial U,\\
    & \mathbf{v}_{_\beta}\to -W\log\frac{1}{\beta}~\mathbf{e}_z\ \  \text{at}\ \infty,
\end{split}
\end{equation*}
where $\mathbf{n}$ is the unit outward normal of $\partial U$. Moreover, as $r\to 0$,
\begin{equation*}
     \frac{1}{r}\frac{\partial\psi_{_\beta}}{\partial z}\to 0,\ \ \frac{1}{r\beta}(\psi_{_\beta})_+\to 0\ \text{and}~ \ \frac{1}{r}\frac{\partial\psi_{_\beta}}{\partial r}\  \text{approaches a finite limit}.
\end{equation*}
\item[(v)]Let the center of vorticity be
\begin{equation*}
  X_\beta=\int_D x\zeta_{_\beta}(x)d\textit{m}_2(x),
\end{equation*}
and define the rescaled version of $\zeta_{_\beta}$ to be
\begin{equation*}
  g_{_\beta}(x)={\beta^2}\zeta_{_\beta}(X_\beta+\beta x).
\end{equation*}
Then every accumulation points as $\beta \to 0^+$, in the weak topology of $L^2$, of $\{g_{_\beta}:\beta>0\}$ are radially nonincreasing functions.
\end{itemize}
\end{theorem}

\begin{remark}
Note that the velocity $W\log\frac{1}{\beta}$ is less than predicted by the  Kelvin-Hicks formula $\eqref{KH}$. This phenomenon is mainly due to the presence of obstacles. It seems that the formula $\eqref{KH}$ should be modified in this case. We do not enter into details here.
\end{remark}

Our last result is on the desingularization of vortices in the whole space.

\begin{theorem}\label{thm3}
Let $D=\Pi$. Then for every $W>0$, $\alpha\ge0$ and all sufficiently small $\beta>0$, there exists a weak solution $(\psi_{_\beta},\xi_{_\beta},\zeta_{_\beta})$ of system \eqref{1-16}-\eqref{1-17} with the following properties:
\begin{itemize}
\item[(i)]For any $p>1$, $0<\gamma<1$, $\psi_\beta\in W^{2,p}_{\text{loc}}(D)\cap C_{\text{loc}}^{1,\gamma}(\bar{D})$ and satisfies
\begin{equation*}
  \mathcal{L}\psi_{_\beta}=\zeta_{_\beta}\ \ \text{a.e.} \ \text{in} \ D.
\end{equation*}

\item[(ii)] $(\psi_{_\beta},\xi_{_\beta},\zeta_{_\beta})$ is of the form
\begin{equation*}
      \psi_{_\beta}=\mathcal{K}\zeta_{_\beta}-\frac{Wr^2}{2}\log{\frac{1}{\beta}}-\mu_{_\beta},\ \ \xi_{_\beta}=\frac{1}{\beta}(\psi_{_\beta})_+,\ \ \zeta_{_\beta}=\frac{1}{r^2\beta^2}{(\psi_{_\beta})}_{+}+\frac{\alpha}{\beta}\chi_{_{\{\psi_{\beta}>0\}}},
\end{equation*}
for some $\mu_{_\beta}>0$ depending on $\beta$. Furthermore,
\begin{equation*}
  \kappa(\zeta_{_\beta})=1\ \  \text{and}\ \ \ \xi_{_\beta} \le \frac{R_0}{\beta},
\end{equation*}
for some positive constant $R_0$ independent of $\beta$.

 \item[(iii)]Let $\Omega_\beta:=supp(\xi_{_\beta})=supp(\zeta_{_\beta})$, there exists some constant $R_1>1$ independent of $\beta$ such that
 \begin{equation*}
   diam(\Omega_\beta)\le R_1\beta.
 \end{equation*}
 Moreover,
 \begin{equation*}
   \lim_{\lambda \to +\infty}dist_{\mathcal{C}_{r_{_*}}}(\Omega_\beta)=0,
\end{equation*}
where $r_*=\frac{1}{4\pi W}$.
Furthermore, as $\beta \to 0^+$,
 $$\mu_{_\beta}=\frac{3}{32}\frac{1}{\pi^2W}\log{\frac{1}{\beta}}+O(1).$$
\item[(iv)] Let $$\mathbf{v}_{_\beta}=\frac{1}{r}\Big(-\frac{\partial\psi_{_\beta}}{\partial z}\mathbf{e}_r+\frac{1}{\beta}(\psi_{_\beta})_+\mathbf{e}_\theta+\frac{\partial\psi_{_\beta}}{\partial r}\mathbf{e}_z\Big).$$ Then
\begin{equation*}
  \mathbf{v}_{_\beta}\to -W\log\frac{1}{\beta}~\mathbf{e}_z\ \  \text{at}\ \infty.
\end{equation*}
 Moreover, as $r\to 0$,
\begin{equation*}
     \frac{1}{r}\frac{\partial\psi_{_\beta}}{\partial z}\to 0,\ \ \frac{1}{r\beta}(\psi_{_\beta})_+\to 0\ \text{and}~ \ \frac{1}{r}\frac{\partial\psi_{_\beta}}{\partial r}\  \text{approaches a finite limit}.
\end{equation*}
\item[(v)]Let the center of vorticity be
\begin{equation*}
  X_\beta=\int_D x\zeta_{_\beta}(x)d\textit{m}_2(x),
\end{equation*}
and define the rescaled version of $\zeta_{_\beta}$ to be
\begin{equation*}
  g_{_\beta}(x)={\beta^2}\zeta_\beta(X_\beta+\beta x).
\end{equation*}
Then every accumulation points as $\beta \to 0^+$, in the weak topology of $L^2$, of $\{g_{_\beta}:\beta>0\}$ are radially nonincreasing functions.
\end{itemize}
\end{theorem}

\begin{remark}
 Compared with \cite{Tur89}, the velocities at infinity of our solutions are determined. Our result is consistent with the Kelvin-Hicks formula \eqref{KH}.
\end{remark}

\begin{remark}
  When $\alpha=0$, we have $\mathbf{v}_{_\beta}=\beta\, \nabla\times\mathbf{v}_{_\beta}$ in the vortex core $\Omega_\beta$. Consequently, the vortex rings we coutructed with $\alpha=0$ consist of a Beltrami flow within their vortex cores.
\end{remark}

\begin{remark}
With these results in hand, one might further show that $K\zeta_{_\beta}$ bifurcates from the Green's function as $\beta$ tends to $0^+$ in all three cases (see \cite{CWZ}). More precisely, let $a(\beta)$ be any point of $\Omega_\beta$. Then, as $\beta \to 0^+$, we have
\begin{equation*}
  \mathcal{K}\zeta_{_\beta}(\cdot)-{K(\cdot,a(\beta))} \to 0 \ \  \text{in}\ \  W^{1,p}_{\text{loc}}(D),\ \ ~1\le p<2,
\end{equation*}
and hence in $L^r_{\text{loc}}(D)$, $1\le r<\infty$. Moreover, for any $\gamma\in (0,1)$, as $\beta \to 0^+$,
\begin{equation*}
  \mathcal{K}\zeta_{_\beta}(\cdot)-{K(\cdot,a(\beta))} \to 0 \ \  \text{in}\ \  C^{1,\gamma}_{loc}(D\backslash\{(r_*,0)\}).
\end{equation*}
\end{remark}

\section{Vortex Rings in a Cylinder}

In this section we will give proof for Theorem \ref{thm1}. To this end, we need to establish some auxiliary results. The first two  in the following are about some estimates of the Green's function of $\mathcal{L}$ in $D$. From \cite{Ta}, we have
\begin{lemma}\label{le1}
  Let D be admissible. Let $K(r,z,r',z')$ be the Green's function of $\mathcal{L}$ in $D$, with respect to zero Dirichlet data and measure $rdrdz$. Then
\begin{equation}\label{200}
  K(r,z,r',z')=G(r,z,r',z')-H(r,z,r',z'),
\end{equation}
where
\begin{equation}\label{000}
   G(r,z,r',z')=\frac{rr'}{4\pi}\int_{-\pi}^{\pi}\frac{\cos\theta'd\theta'}{[(z-z')^2+r^2+r'^2-2rr'\cos\theta']^\frac{1}{2}},
\end{equation}
and  $H(r,z,r',z')\in C^{\infty}(D\times D)$ is non-negative. Moreover, define
\begin{equation}\label{201}
 \sigma=[(r-r')^2+(z-z')^2]^\frac{1}{2}/(4rr')^\frac{1}{2},
\end{equation}
then for all $\sigma > 0$
\begin{equation}\label{Tadiewrong}
  0<K(r,z,r',z')\leq G(r,z,r',z')\leq\frac{(rr')^\frac{1}{2}}{4\pi}\sinh^{-1}(\frac{1}{\sigma}).
\end{equation}
\end{lemma}

For $G$ defined by \eqref{000}, we have the following expansion.
\begin{lemma}\label{le2}
  Let $\sigma$ be defined by \eqref{201}, then there exists a continuous function $f\in L^{\infty}(D\times D)$ such that
\begin{equation}\label{202}
   G(r,z,r',z')=\frac{\sqrt{rr'}}{2\pi}\log{\frac{1}{\sigma}}+\frac{\sqrt{rr'}}{2\pi}\log(1+\sqrt{\sigma^2+1})+f(r,z,r',z')\sqrt{rr'},\ \ \text{in}\ D\times D.
\end{equation}
\end{lemma}

\begin{proof}
Let $\sigma$ be defined by $\eqref{201}$, then
\begin{equation}\label{203}
\begin{split}
    G(r,z,r',z') & =\frac{rr'}{2\pi}\int_{0}^{\pi}\frac{\cos\theta'd\theta'}{[(z-z')^2+r^2+r'^2-2rr'\cos\theta']^\frac{1}{2}}\\
                 & =\frac{\sqrt{rr'}}{2\pi}\int_{0}^{\pi}\frac{\cos\theta'd\theta'}{\{[(z-z')^2+(r-r')^2]/(rr')+2(1-\cos\theta')\}^\frac{1}{2}}\\
                 & =\frac{\sqrt{rr'}}{2\pi}\int_{0}^{\pi}\frac{[\cos\theta'/2+(\cos\theta'-\cos\theta'/2)]d\theta'}{[4\sigma^2+4(\sin\theta'/2)^2]^\frac{1}{2}}.
\end{split}
\end{equation}
Notice that
\begin{equation}\label{204}
  \int_{0}^{\pi}\frac{\cos\theta'/2d\theta'}{[4\sigma^2+4(\sin\theta'/2)^2]^\frac{1}{2}}=\log(1+\sqrt{\sigma^2+1})+\log{\frac{1}{\sigma}},
\end{equation}
and
\begin{equation}\label{205}
  \int_{0}^{\pi}\frac{|\cos\theta'-\cos\theta'/2|d\theta'}{[4\sigma^2+4(\sin\theta'/2)^2]^\frac{1}{2}}\le \text{const.}<+\infty.
\end{equation}
From \eqref{203},\eqref{204} and \eqref{205}, $\eqref{202}$ clearly follows.
\end{proof}

The following result is a variant of Lemma 6 of Burton \cite{B5}, which can be proved by almost the same arguments.
\begin{lemma}\label{burton}
  Let $D\subseteq \Pi$ be a domain, let $(\psi,\zeta)\in W^{2,p}_{\text{loc}}(D)\times L^\infty(D)$ for some $p>1$ satisfying $\mathcal{L}\psi=\zeta$ a.e. in $D$. Suppose that $\zeta=f\circ \psi$ a.e. in $D$, for some monotonic function $f$. Then $\text{div}~(\zeta \nabla^\bot\psi)=0$ in the sense of distribution.
\end{lemma}

In the rest of this section we let $D=\{(r,z)\in \Pi~|~0<r<d\}$ for some $d>0$.

Inspired by the variational principle proposed by Turkington (see (VP2) of \cite{Tur89}), we consider the following specific constrained maximization problem in $\zeta$, the solutions of which may furnish the desired steady vortex rings with swirl. Notice that the energy functional depends on parameter $\beta$ and the variational admissible class adopted here is different from that in \cite{Tur89}. This process allows the velocity of the flow at far fields to be determined, rather than being a Lagrange multiplier, as in \cite{Tur89}. Furthermore, notice that we require $\zeta \in L^\infty$ in the admissible class. As one will see, this requirement is crucial in the proof.

For fixed $\alpha \ge0$, $\beta>0$ and $W>0$, we consider the energy functional as follows
\[E_\beta(\zeta)=\frac{1}{2}\int_D{\zeta \mathcal{K}\zeta}d\nu-\frac{{W}}{2}\log{\frac{1}{\beta}}\int_{D}r^2\zeta d\nu-\frac{\beta^2}{2}\int_D r^2(\zeta-\frac{\alpha}{\beta})_{+}^2d\nu.\]
We adopt the class of admissible functions $\mathcal{A}_{\beta,\Lambda}$ as follows
\begin{equation*}
\mathcal{A}_{\beta,\Lambda}=\{\zeta\in L^\infty(D)~|~ 0\le \zeta \le \frac{\Lambda}{\beta^2}~ \text{a.e.}, \int_{D}\zeta d\nu \le1\},
\end{equation*}
where $\Lambda>\max\{\alpha \beta,1\}$ is a fixed positive number. For any $\zeta\in \mathcal{A}_{\beta,\Lambda}$, one can check that $\mathcal{K}\zeta \in W^{2,p}_{\text{loc}}(D)\cap C^{1,\gamma}(\bar{D})$ for any $p>1$, $0<\gamma<1$ (see \cite{B1}).

Let ${\zeta}^*$ be the Steiner symmetrization of $\zeta$ with respect to the line $z=0$ in $D$ (see Appendix I of \cite{FB}). An absolute maximum for $E_\beta$ over $\mathcal{A}_{\beta,\Lambda}$  can be easily found.

\begin{lemma}\label{le3}
For every prescribed $\alpha \ge0$, $1>\beta>0$, $W>0$ and $\Lambda>\max\{\alpha \beta,1\}$, there exists $\zeta=\zeta_{_{\beta,\Lambda}} \in \mathcal{A}_{\beta,\Lambda}$ such that
\begin{equation}\label{206}
 E_\beta(\zeta_{_{\beta,\Lambda}})= \max_{\tilde{\zeta} \in \mathcal{A}_{\beta,\Lambda}}E_\beta(\tilde{\zeta})<+\infty.
\end{equation}
Moreover, there exists a Lagrange multiplier $\mu_{_{\beta,\Lambda}} \ge 0$ such that
\begin{equation}\label{207}
\zeta_{_{\beta,\Lambda}}=\Big(\frac{\psi_{_{\beta,\Lambda}}}{r^2\beta^2}+\frac{\alpha}{\beta}\Big){\chi}_{\{0<\psi_{_{\beta,\Lambda}}<(\Lambda-\alpha\beta)r^2\}}+\frac{\Lambda}{\beta^2}\chi_{\{\psi_{_{\beta,\Lambda}} \ge (\Lambda-\alpha\beta)r^2\}} \ \ a.e.\  \text{in}\  D,
\end{equation}
where
\begin{equation}\label{001}
 \psi_{_{\beta,\Lambda}}=\mathcal{K}\zeta_{_{\beta,\Lambda}}-\frac{Wr^2}{2}\log{\frac{1}{\beta}}-\mu_{_{\beta,\Lambda}}.
\end{equation}
Also, $\zeta_{_{\beta,\Lambda}}=(\zeta_{_{\beta,\Lambda}})^*$. Furthermore, whenever $E_\beta(\zeta_{_{\beta,\Lambda}})>0$ and $\mu_{_{\beta,\Lambda}}>0$ there holds $\int_D\zeta_{_{\beta,\Lambda}} d\nu=1$ and $\zeta_{_{\beta,\Lambda}}$ has compact support in $D$.
\end{lemma}

\begin{proof}
We may take a sequence $\zeta_{j}\in \mathcal{A}_{\beta,\Lambda}$ such that as $j\to +\infty$
\begin{equation*}
  \begin{split}
        E_\beta(\zeta_{j}) & \to \sup\{E_\beta(\tilde{\zeta})~|~\tilde{\zeta}\in \mathcal{A}_{\beta,\Lambda}\}, \\
        \zeta_{j} & \to \zeta\in L^{10/7}({D,\nu})~~\text{weakly}.
  \end{split}
\end{equation*}
It is easily checked that $\zeta\in \mathcal{A}_{\beta,\Lambda}$. Using the standard arguments (see Lemma 4.6 of \cite{Dou2}), we may assume that $\zeta_{j}=(\zeta_{j})^*$, and hence $\zeta=(\zeta)^*$. By Lemma 4.8 of \cite{Dou2}, we first have
\begin{equation*}
      \lim_{j\to +\infty}\int_D{\zeta_{j} \mathcal{K}\zeta_{j}}d\nu = \int_D{\zeta \mathcal{K}\zeta}d\nu,\ \text{as}\ j\to +\infty.
\end{equation*}
On the other hand, we have the lower semicontinuity of the rest of terms, namely,
\begin{equation*}
  \begin{split}
    \liminf_{j\to +\infty} \int_{D}r^2\zeta_j d\nu  & \ge  \int_{D}r^2\zeta d\nu, \\
    \liminf_{j\to +\infty}\int_D r^2(\zeta_j-\frac{\alpha}{\beta})_{+}^2d\nu   & \ge \int_D r^2(\zeta-\frac{\alpha}{\beta})_{+}^2d\nu.
  \end{split}
\end{equation*}
Consequently, we may conclude that $E_\beta(\zeta)=\lim_{j\to +\infty}E_\beta(\zeta_j)=\sup E_\beta$, with $\zeta\in \mathcal{A}_{\beta,\Lambda}$.

We now turn to prove $\eqref{207}$. Note that we may assume $\zeta \not\equiv 0$, otherwise the conclusion is obtained by letting $\mu_{_{\beta,\Lambda}}=0$.
We consider the family of variations of $\zeta$
\begin{equation*}
  \zeta_{(s)}=\zeta+s(\tilde{\zeta}-\zeta),\ \ \ s\in[0,1],
\end{equation*}
defined for arbitrary $\tilde{\zeta}\in \mathcal{A}_{\beta,\Lambda}$. Since $\zeta$ is a maximizer, we have
\begin{equation*}
  \begin{split}
     0 & \ge \frac{d}{ds}E(\zeta_{(s)})|_{s=0^+} \\
       & =\int_{D}(\tilde{\zeta}-\zeta)\big(\mathcal{K}\zeta-\frac{Wr^2}{2} \log{\frac{1}{\beta}}-\beta^2r^2(\zeta-\frac{\alpha}{\beta})_+\big)d\nu.
  \end{split}
\end{equation*}
This implies that for any  $\tilde{\zeta}\in \mathcal{A}_{\beta,\Lambda}$, there holds
\begin{equation*}
  \int_{D}\zeta \big(\mathcal{K}\zeta-\frac{Wr^2}{2} \log{\frac{1}{\beta}}-\beta^2r^2(\zeta-\frac{\alpha}{\beta})_+\big)d\nu \ge \int_{D}\tilde{\zeta}  \big(\mathcal{K}\zeta-\frac{Wr^2}{2} \log{\frac{1}{\beta}}-\beta^2r^2(\zeta-\frac{\alpha}{\beta})_+\big)d\nu.
\end{equation*}
By an adaptation of the bathtub principle \cite{Lieb}, we obtain
\begin{equation}\label{208}
  \begin{split}
    \mathcal{K}\zeta-\frac{Wr^2}{2} \log{\frac{1}{\beta}}-\mu_{_{\beta, \Lambda}} &\ge \beta^2r^2(\zeta-\frac{\alpha}{\beta})_+ \ \ \  whenever\  \zeta=\frac{\Lambda}{\beta^2}, \\
     \mathcal{K}\zeta-\frac{Wr^2}{2} \log{\frac{1}{\beta}}-\mu_{_{\beta, \Lambda}} &=\beta^2r^2(\zeta-\frac{\alpha}{\beta})_+ \ \ \  whenever\  0<\zeta<\frac{\Lambda}{\beta^2}, \\
    \mathcal{K}\zeta-\frac{Wr^2}{2} \log{\frac{1}{\beta}}-\mu_{_{\beta, \Lambda}} & \le \beta^2r^2(\zeta-\frac{\alpha}{\beta})_+ \ \ \  whenever\  \zeta=0,
  \end{split}
\end{equation}
where $$\mu_{_{\beta, \Lambda}}=\inf\{t:|\{\mathcal{K}\zeta-\frac{Wr^2}{2} \log{\frac{1}{\beta}}-\beta^2r^2(\zeta-\frac{\alpha}{\beta})_+>t\}|\le 1\}\in \mathbb{R}_+.$$
It follows that $\{0<\zeta\le \alpha/\beta\}\subseteq\{\mathcal{K}\zeta-\frac{Wr^2}{2} \log{\frac{1}{\beta}}=\mu_{_{\beta, \Lambda}}\}$. Since $\mathcal{L}(\mathcal{K}\zeta-\frac{Wr^2}{2} \log{\frac{1}{\beta}})=\zeta$ almost everywhere in $D$, we conclude that $m_2(\{0<\zeta\le \alpha/\beta\})=0$. Now the stated form $\eqref{207}$ follows immediately.

When $E_\beta(\zeta)>0$ and $\mu_{_{\beta,\Lambda}}>0$, we first prove $\int_D\zeta d\nu=1$. Suppose not, then we may consider the family of variations of $\zeta$
\begin{equation*}
  \zeta_{(s)}=\zeta+s\phi,\ \ \ s>0,
\end{equation*}
defined for arbitrary $\phi \in L^1\cap L^{\infty}(D)$ satisfying $\phi\ge 0$ a.e. on $\{\zeta<\delta\}$ and $\phi\le 0$ a.e. on $\{\zeta>\Lambda/\beta^2-\delta\}$ for some $\delta>0$. Clearly, $\zeta_{(s)}\in \mathcal{A}_{\beta,\Lambda}$ for all sufficiently small $s>0$.
We then have
\begin{equation*}
  \begin{split}
     0 & \ge \frac{d}{ds}E(\zeta_{(s)})|_{s=0^+} \\
       & =\int_{D}\phi(\mathcal{K}\zeta-\frac{Wr^2}{2} \log{\frac{1}{\beta}}-\beta^2r^2(\zeta-\frac{\alpha}{\beta})_+)d\nu,
  \end{split}
\end{equation*}
which implies that
\begin{equation}\label{209}
  \begin{split}
         \mathcal{K}\zeta-\frac{Wr^2}{2} \log{\frac{1}{\beta}} &\ge \beta^2r^2(\zeta-\frac{\alpha}{\beta})_+ \ \ \  whenever\  \zeta=\frac{\Lambda}{\beta^2}, \\
     \mathcal{K}\zeta-\frac{Wr^2}{2} \log{\frac{1}{\beta}} &=\beta^2r^2(\zeta-\frac{\alpha}{\beta})_+ \ \ \  whenever\  0<\zeta<\frac{\Lambda}{\beta^2}, \\
    \mathcal{K}\zeta-\frac{Wr^2}{2} \log{\frac{1}{\beta}} & \le \beta^2r^2(\zeta-\frac{\alpha}{\beta})_+ \ \ \  whenever\  \zeta=0.
  \end{split}
\end{equation}
Combining $\eqref{208}$ and $\eqref{209}$, we conclude that
\begin{equation*}
  \{\zeta>0\}=\{\mathcal{K}\zeta-\frac{Wr^2}{2} \log{\frac{1}{\beta}}>0\}=\{ \mathcal{K}\zeta-\frac{Wr^2}{2} \log{\frac{1}{\beta}}>\mu_{_{\beta, \Lambda}}\},
\end{equation*}
which is clearly a contradiction. Note that $\mathcal{K}\zeta=0$ on $\partial D$ and $\mathcal{K}\zeta$ is Steiner-symmetric, it is not hard to prove that $\zeta$ has compact support in $D$ (\,see, e.g. \cite{BB, FT, Tur89}\,). The proof is thus completed.
\end{proof}

Our focus is next the asymptotic behaviour of $\zeta_{_{\beta,\Lambda}}$ when $\beta \to 0^+$. In the sequel we shall denote $C,C_1,C_2, ..., $ for positive constants independent of $\beta$.

The following lemma gives a lower bound of the energy.
\begin{lemma}\label{le4}
  For any  $a\in(0,d)$, there exists $C>0$ such that for all $\beta$ sufficiently small, we have
\begin{equation*}
  E_\beta(\zeta_{_{\beta,\Lambda}})\ge (\frac{a}{4\pi}-\frac{Wa^2}{2})\log{\frac{1}{\beta}}-C,
\end{equation*}
where the positive number $C$ depends on $a$, but not on $\beta$, $\Lambda$.
\end{lemma}

\begin{proof}
Choose a test function $\tilde{\zeta} \in \mathcal{A}_{\beta,\Lambda}$ defined by $\tilde{\zeta}=\frac{1}{\beta^2} \chi_{_{B_{{\beta}/{\surd(a\pi)}}((a,0))}}$. Since $\zeta_{_{\beta,\Lambda}}$ is a maximizer, we have $E_\beta(\zeta_{_{\beta,\Lambda}})\ge E_\beta(\tilde{\zeta})$. By Lemmas $\ref{le1}$ and $\ref{le2}$, we obtain
 \begin{equation*}
\begin{split}
    E_\beta(\tilde{\zeta})&= \frac{1}{2}\iint_{D\times D}\tilde{\zeta}(r,z)K(r,z,r',z')\tilde{\zeta}(r',z')r'rdr'dz'drdz-{\frac{W}{2} \log{\frac{1}{\beta}}}\int_{D}r^2\tilde{\zeta}d\nu\\
                    &\ \ \ \ \ -\frac{\beta^2}{2}\int_D r^2(\tilde{\zeta}-\frac{\alpha}{\beta})_{+}^2d\nu\\
                    &\ge \frac{a+O(\beta)}{4\pi}\iint_{D\times D}\log{\frac{1}{[(r-r')^2+(z-z')^2]^{1/2}}}\tilde{\zeta}(r,z)\tilde{\zeta}(r',z')r'rdr'dz'drdz\\
                    &\ \ \ \ \ -\frac{W[a+O(\beta)]^2}{2}\log{\frac{1}{\beta}}-\frac{\beta^2}{2}\int_Dr^2{\tilde{\zeta}}^2d\nu-C_1\\
                    &\ge (\frac{a}{4\pi}-\frac{Wa^2}{2})\log{\frac{1}{\beta}}-C.
\end{split}
\end{equation*}
We therefore complete the proof.
\end{proof}

We now turn to estimate the Lagrange multiplier $\mu_{_{\beta,\Lambda}}$.
\begin{lemma}\label{ad1}
  For all sufficiently small $\beta$, there holds
  \begin{equation}\label{0-1}
    \mu_{_{\beta,\Lambda}}=2E_\beta(\zeta_{_{\beta,\Lambda}})+\frac{W}{2}\log\frac{1}{\beta}\int_D\zeta_{_{\beta,\Lambda}} r^2 d\nu +O(1),
  \end{equation}
 where the bounded quantity $O(1)$ does not depend on $\beta$ and $\Lambda$.
\end{lemma}

\begin{proof}
  Recalling \eqref{208}, we have
\begin{equation*}
  \begin{split}
     2E_\beta(\zeta_{_{\beta,\Lambda}}) &= \int_D{\zeta_{_{\beta,\Lambda}} \mathcal{K}\zeta_{_{\beta,\Lambda}}}d\nu-{W\log{\frac{1}{\beta}}}\int_{D}r^2\zeta_{_{\beta,\Lambda}} d\nu-\int_D r^2\beta^2(\zeta_{_{\beta,\Lambda}}-\frac{\alpha}{\beta})_{+}^2d\nu\\
       &= \int_D{\zeta_{_{\beta,\Lambda}} \big(\mathcal{K}\zeta_{_{\beta,\Lambda}}-\frac{W}{2}r^2\log{\frac{1}{\beta}}-\mu_{_{\beta,\Lambda}}\big)}d\nu-{\frac{W}{2}\log{\frac{1}{\beta}}}\int_{D}r^2 \zeta_{_{\beta,\Lambda}} d\nu\\
       &\ \ \ \ -\int_D r^2\beta^2(\zeta_{_{\beta,\Lambda}}-\frac{\alpha}{\beta})_{+}^2d\nu+\mu_{_{\beta,\Lambda}} \int_D\zeta_{_{\beta,\Lambda}} d\nu \\
       &= \int_{\{\frac{\alpha}{\beta}<\zeta_{_{\beta,\Lambda}}<\frac{\Lambda}{\beta^2} \}}{\zeta_{_{\beta,\Lambda}} r^2\beta^2(\zeta_{_{\beta,\Lambda}}-\frac{\alpha}{\beta})_+}d\nu+\int_{\{\zeta_{_{\beta,\Lambda}}=\frac{\Lambda}{\beta^2}\}}\zeta_{_{\beta,\Lambda}}\psi_{\beta,\Lambda}d\nu \\
       &\ \ \ \ -{\frac{W}{2}\log{\frac{1}{\beta}}}\int_{D}r^2\zeta_{_{\beta,\Lambda}} d\nu-\int_D r^2\beta^2(\zeta_{_{\beta,\Lambda}}-\frac{\alpha}{\beta})_{+}^2d\nu+\mu_{_{\beta,\Lambda}} \int_D\zeta_{_{\beta,\Lambda}} d\nu\\
       &= \int_D{\alpha \beta r^2(\zeta_{_{\beta,\Lambda}}-\frac{\alpha}{\beta})_+}d\nu+\int_{\{\zeta_{_{\beta,\Lambda}}=\frac{\Lambda}{\beta^2} \}}\zeta_{_{\beta,\Lambda}} \big[\psi_{_{\beta,\Lambda}}-(\Lambda-\alpha\beta)r^2\big]_+d\nu\\
       &\ \ \ \ -{\frac{W}{2}\log{\frac{1}{\beta}}}\int_{D}r^2\zeta_{_{\beta,\Lambda}} d\nu+\mu_{_{\beta,\Lambda}} \int_D\zeta_{_{\beta,\Lambda}} d\nu.
  \end{split}
\end{equation*}
Hence
\begin{equation}\label{0-2}
  \mu_{_{\beta,\Lambda}}\int_D\zeta_{_{\beta,\Lambda}} d\nu=2E_\beta(\zeta_{_{\beta,\Lambda}})+\frac{W}{2}\log\frac{1}{\beta}\int_D\zeta_{_{\beta,\Lambda}} r^2 d\nu -\int_D{\alpha \beta r^2(\zeta_{_{\beta,\Lambda}}-\frac{\alpha}{\beta})_+}d\nu-\frac{\Lambda}{\beta^2}\int_D U_{\beta,\Lambda}d\nu,
\end{equation}
where $U_{\beta,\Lambda}:=\big[\psi_{_{\beta,\Lambda}}-(\Lambda-\alpha\beta)r^2\big]_+$. Next, we prove that the last two terms on the right hand side of \eqref{0-2} are uniformly bounded with respect to $\beta$ and $\Lambda$. Recall that
\begin{equation*}
  \mathcal{L}\psi_{_{\beta,\Lambda}}=\big\{\frac{\psi_{_{\beta,\Lambda}}}{r^2\beta^2}+\frac{\alpha}{\beta}\big\}{\chi}_{_{\{0<\psi_{_{\beta,\Lambda}}<(\Lambda-\alpha\beta)r^2\}}}+\frac{\Lambda}{\beta^2}\chi_{_{\{\psi \ge (\Lambda-\alpha\beta)r^2\}}}.
\end{equation*}
If we take $U_{\beta,\Lambda}\in H(D)$ as a test function, we obtain
\begin{equation}\label{0-3}
\begin{split}
   \int_D\frac{|\nabla U_{\beta,\Lambda}|^2}{r^2} d\nu&= \frac{\Lambda}{\beta^2}\int_D U_{\beta,\Lambda} d\nu \\
     & \le \frac{\Lambda}{\beta^2}|\{U_{\beta,\Lambda}\ge 0\}|^{\frac{1}{2}}\Big(\int_D U_{\beta,\Lambda}^2 d\nu \Big)^{\frac{1}{2}}\\
     & \le \frac{\Lambda d^{\frac{1}{2}}}{\beta^2}|\{U_{\beta,\Lambda}\ge 0\}|^{\frac{1}{2}}\Big(\int_D U_{\beta,\Lambda}^2 drdz \Big)^{\frac{1}{2}}\\
     & \le \frac{C\Lambda d^{\frac{1}{2}}}{\beta^2}|\{U_{\beta,\Lambda}\ge 0\}|^{\frac{1}{2}}\Big(\int_D |\nabla U_{\beta,\Lambda}| drdz \Big)\\
     &  \le \frac{C\Lambda d^{\frac{1}{2}}}{\beta^2}|\{U_{\beta,\Lambda}\ge 0\}|\Big(\int_D \frac{|\nabla U_{\beta,\Lambda}|^2 }{r^2}d\nu \Big)^{\frac{1}{2}}\\
     & \le Cd^{\frac{1}{2}}\Big(\int_D \frac{|\nabla U_{\beta,\Lambda}|^2 }{r^2}d\nu \Big)^{\frac{1}{2}}.
\end{split}
\end{equation}
where we used H$\ddot{\text{o}}$der's inequality and Sobolev's inequality, and the positive constant $C$  does not depend on $\beta$ and $\Lambda$. From \eqref{0-3} we conclude that the rightmost term of \eqref{0-2} is uniformly bounded with respect to $\beta$ and $\Lambda$. Finally, note that
\begin{equation*}
  \int_D{\alpha \beta r^2(\zeta_{_{\beta,\Lambda}}-\frac{\alpha}{\beta})_+}d\nu\le \alpha \beta d^2 \int_D \zeta_{_{\beta,\Lambda}}d\nu \le \alpha d^2.
\end{equation*}
Therefore
\begin{equation}\label{0-0}
   \mu_{_{\beta,\Lambda}}\int_D\zeta_{_{\beta,\Lambda}} d\nu=2E_\beta(\zeta_{_{\beta,\Lambda}})+\frac{W}{2}\log\frac{1}{\beta}\int_D\zeta_{_{\beta,\Lambda}} r^2 d\nu +O(1),
\end{equation}
where the bounded quantity $O(1)$ does not depend on $\beta$ and $\Lambda$. From \eqref{0-0} and Lemma \ref{le4}, we conclude that $E_\beta(\zeta_{_{\beta,\Lambda}})>0$ and $\mu_{_{\beta,\Lambda}}>0$ when $\beta$ is small enough. This further implies that $\int_D\zeta_{_{\beta,\Lambda}} d\nu=1$ for all sufficiently small $\beta$. Hence we obtain \eqref{0-1}, and the proof is completed.
\end{proof}

From Lemmas \ref{le4} and \ref{ad1}, we immediately have the following estimate.
\begin{lemma}\label{le5}
For any  $a\in(0,{d})$, there exists $C>0$ such that for all $\beta$ sufficiently small, we have
\begin{equation}\label{211}
  \mu_{_{\beta,\Lambda}}\ge (\frac{a}{2\pi}-{Wa^2})\log{\frac{1}{\beta}}+\frac{W}{2}\log{\frac{1}{\beta}}\int_D\zeta_{_{\beta,\Lambda}} r^2d\nu-C.
\end{equation}
where the positive number $C$ depends on $a$, but not on $\beta$, $\Lambda$.
\end{lemma}

We now introduce the function $\Gamma_1$ as follows
\begin{equation*}
  \Gamma_1(t)=\frac{t}{2\pi}-Wt^2,\ t\in[0,d].
\end{equation*}
Set
\[
r_*=\left\{
   \begin{array}{lll}
        \frac{1}{4\pi W} &    \text{if} & W>{1}/{(4\pi d)}, \\
         d                  &    \text{if} & W\le {1}/{(4\pi d)}.
    \end{array}
   \right.
\]
It is easy to check that $\Gamma_1(r_*)=\max_{t\in[0,d]}\Gamma_1(t)$.
Let $A_\beta=\inf\{r|(r,0)\in {supp}(\zeta_{_{\beta,\Lambda}})\}$,\ $B_\beta=\sup\{r|(r,0)\in {supp}(\zeta_{_{\beta,\Lambda}})\}$.

Next, we show that in order to maximize the energy, these vortices constructed above
must be concentrated as $\beta$ tends to zero. We reach our goal by several steps as follows.
\begin{lemma}\label{le6}
  $\lim_{\beta\to 0^+}A_\beta=r_*$.
\end{lemma}

\begin{proof}
  Let $\sigma$ be defined by $\eqref{201}$ and $\gamma\in(0,1)$. Let $(r_\beta,z_\beta)\in {supp}(\zeta_{_{\beta,\Lambda}})$, then we have $$\mathcal{K}\zeta_{_{\beta,\Lambda}}(r_\beta,z_\beta)-\frac{W(r_\beta)^2}{2}\log{\frac{1}{\beta}}\ge \mu_{_{\beta,\Lambda}}.$$ By Lemma $\ref{le1}$, we know that
\begin{equation*}
\begin{split}
   \mathcal{K}\zeta_{_{\beta,\Lambda}}(r_\beta,z_\beta)&\le \int_{D}G(r_\beta,z_\beta,r',z')\zeta_{_{\beta,\Lambda}}(r',z')r'dr'dz'\\
                       &=\big(\int_{D\cap\{\sigma>\beta^\gamma\}}+\int_{D\cap\{\sigma \le \beta^\gamma\}}\big)G(r_\beta,z_\beta,r',z')\zeta_{_{\beta,\Lambda}}(r',z') r'dr'dz'\\
                       &:=I_1+I_2.
\end{split}
\end{equation*}

For the first term $I_1$, we can use $\eqref{Tadiewrong}$ to obtain
\begin{equation}\label{212}
\begin{split}
   I_1 &=\int\limits_{D\cap\{\sigma>\beta^\gamma\}}G(r_\beta,z_\beta,r',z')\zeta_{_{\beta,\Lambda}}(r',z') r'dr'dz'\\
       &\le \frac{(r_\beta)^{\frac{1}{2}}}{2\pi} \sinh^{-1}(\frac{1}{\beta^\gamma})\int\limits_{D\cap\{\sigma>\beta^\gamma\}}\zeta_{_{\beta,\Lambda}}(r',z')r'^{{3}/{2}}dr'dz'\\
       &\le \frac{d}{2\pi} \sinh^{-1}(\frac{1}{\beta^\gamma})\int\limits_{D\cap\{\sigma>\beta^\gamma\}}\zeta_{_{\beta,\Lambda}}(r',z')r'dr'dz'\\
       &\le \frac{d}{2\pi} \sinh^{-1}(\frac{1}{\beta^\gamma}).
\end{split}
\end{equation}
On the other hand, noticing that $D\cap\{\sigma \le \beta^\gamma\} \subseteq B_{2d\beta^\gamma}\big((r_\beta,z_\beta)\big)$, hence using Lemma $\ref{le2}$ we get
\begin{equation}\label{213}
\begin{split}
   I_2 &=\int\limits_{D\cap\{\sigma \le \beta^\gamma\}}G(r_\beta,z_\beta,r',z')\zeta_{_{\beta,\Lambda}}(r',z')r'dr'dz'\\
       &\le \frac{(r_\beta)^2+C\beta^\gamma}{2\pi}\int\limits_{D\cap\{\sigma \le \beta^\gamma\}}\log [(r_\beta-r')^2+(z_\beta-z')^2]^{-\frac{1}{2}}\zeta_{_{\beta,\Lambda}}(r',z')dr'dz'+C\\
       &\le \frac{(r_\beta)^2+C\beta^\gamma}{2\pi}\log\frac{1}{\beta} \int\limits_{D\cap\{\sigma \le \beta^\gamma\}}\zeta_{_{\beta,\Lambda}}(r',z')dr'dz'+C\\
       &\le \frac{(r_\beta)^2}{2\pi}\log\frac{1}{\beta} \int\limits_{D\cap\{\sigma \le \beta^\gamma\}}\zeta_{_{\beta,\Lambda}}(r',z')dr'dz'+C\\
       &\le \frac{r_\beta}{2\pi}\log\frac{1}{\beta} \int\limits_{B_{2d\beta^\gamma}\big((r_\beta,z_\beta)\big)}\zeta_{_{\beta,\Lambda}} d\nu+C,
\end{split}
\end{equation}
where we have used an easy rearrangement inequality. Therefore, by Lemma $\ref{le5}$, we conclude that for any $a\in(0,d)$ there holds
\begin{equation*}
\begin{split}
   \frac{r_\beta}{2\pi}\log\frac{1}{\beta} \int\limits_{B_{2d\beta^\gamma}\big((r_\beta,z_\beta)\big)}\zeta_{_{\beta,\Lambda}} d\nu&+\frac{d}{2\pi} \sinh^{-1}(\frac{1}{\beta^\gamma})-\frac{W(r_\beta)^2}{2}\log{\frac{1}{\beta}} \\
     & \ge (\frac{a}{2\pi}-{Wa^2})\log{\frac{1}{\beta}}+\frac{W}{2}\log{\frac{1}{\beta}}\int_D\zeta_{_{\beta,\Lambda}} r^2d\nu-C.
\end{split}
\end{equation*}
Divide both sides of the above inequality by $\log\frac{1}{\beta}$, we obtain
\begin{equation}\label{214}
\begin{split}
   \Gamma_1(r_\beta)\ge \frac{r_\beta}{2\pi}\int\limits_{B_{2d\beta^\gamma}\big((r_\beta,z_\beta)\big)}\zeta_{_{\beta,\Lambda}} d\nu-W(r_\beta)^2 \ge &~ \Gamma_1(a)+\frac{W}{2}\big{\{}\int_D\zeta_{_{\beta,\Lambda}} r^2d\nu-(r_\beta)^2\big{\}}   \\
     &-\frac{d}{2\pi} \sinh^{-1}(\frac{1}{\beta^\gamma})/\log\frac{1}{\beta}-C/\log\frac{1}{\beta}.
\end{split}
\end{equation}
Notice that
\begin{equation*}
  \int_D\zeta_{_{\beta,\Lambda}} r^2d\nu \ge (A_\beta)^2.
\end{equation*}
Taking $(r_\beta,z_\beta)=(A_\beta,0)$ and letting $\beta$ tend to $0^+$, we deduce from $\eqref{214}$ that
\begin{equation}\label{215}
  \liminf_{\beta\to 0^+}\Gamma_1(A_\beta)\ge \Gamma_1(a)-\gamma d/(2\pi).
\end{equation}
Hence we get the desired result by letting $a\to r_*$ and $\gamma \to 0$.
\end{proof}

\begin{lemma} \label{le7}
As $\beta\to 0^+$, $\int_D\zeta_{_{\beta,\Lambda}} r^2d\nu\to r_*^2$.
\end{lemma}

\begin{proof}
From $\eqref{214}$, we know that for any $\gamma\in(0,1)$,
\begin{equation*}
  0\le \liminf_{\beta\to 0^+}\big[\int_D\zeta_{_{\beta,\Lambda}} r^2d\nu-(A_\beta)^2\big]\le \limsup_{\beta\to 0^+}\big[\int_D\zeta_{_{\beta,\Lambda}} r^2d\nu-(A_\beta)^2\big]\le \frac{d\gamma}{\pi W}.
\end{equation*}
Thus
\begin{equation*}
  \lim_{\beta\to 0^+}\int_D\zeta_{_{\beta,\Lambda}} r^2d\nu =r_*^2.
\end{equation*}
\end{proof}

From Lemma $\ref{le6}$, we immediately get the following result.
\begin{lemma}\label{le7}
For any $\eta>0$, there holds
\begin{equation*}
  \lim_{\beta\to 0^+}\int_{D\cap\{r\ge r_*+\eta\}}\zeta_{_{\beta,\Lambda}} d\nu=0
\end{equation*}
\end{lemma}

\begin{lemma}\label{le8}
$\lim_{\beta\to 0^+}B_\beta=r_*$. Consequently, $\lim_{\beta\to 0^+}dist_{C_{r_{_*}}}\big({supp}(\zeta_{_{\beta,\Lambda}})\big)=0$.
\end{lemma}

\begin{proof}
We may assume $0<r_*<d$, otherwise we have done by Lemma $\ref{le5}$. From $\eqref{214}$, we obtain
\begin{equation*}
\begin{split}
   \frac{d}{2\pi}\liminf_{\beta\to 0^+}\int\limits_{B_{2d\beta^\gamma}\big((B_\beta,0)\big)}\zeta_{_{\beta,\Lambda}} d\nu  & \ge \Gamma_1(r_*)+\frac{W}{2}\liminf_{\beta\to 0^+}(B_\beta)^2+\frac{Wr_*^2}{2}-\frac{d\gamma}{2\pi} \\
     & \ge \frac{r_*}{2\pi}-\frac{d\gamma}{2\pi}.
\end{split}
\end{equation*}
Hence
\begin{equation*}
  \liminf_{\beta\to 0^+}\int\limits_{B_{2d\beta^\gamma}\big((B_\beta,0)\big)}\zeta_{_{\beta,\Lambda}} d\nu\ge\frac{r_*}{d}-\gamma.
\end{equation*}
The desired result clearly follows from Lemma $\ref{le7}$ if we take $\gamma$ so small such that $r_*/d-\gamma>0$.
\end{proof}

\begin{lemma}\label{diam1}
For any number $\gamma\in(0, 1)$, there holds $diam\big({supp}(\zeta_{_{\beta,\Lambda}})\big) \le 4d \beta^{\gamma}$ provided $\beta$ is small enough.
\end{lemma}

\begin{proof}
  Let us use the same notation as in the proof of Lemma $\ref{le6}$. Recalling that $\int_{D}\zeta_{_{\beta,\Lambda}} d\nu=1, $ so it suffices to prove that
  \begin{equation*}
  \int\limits_{B_{2d\beta^\gamma}((r_\beta,z_\beta))}\zeta_{_{\beta,\Lambda}} d\nu>1/2, \ \ \forall\, (r_\beta,z_\beta)\in {supp}(\zeta_{_{\beta,\Lambda}}).
  \end{equation*}
  Firstly, from Lemma $\ref{le8}$ we know that $r_\beta\to r_*$  as $\beta \to 0^+$.
  By $\eqref{214}$ we get
\begin{equation}\label{216}
  \liminf_{\beta\to 0^+}\int\limits_{B_{2d\beta^\gamma}\big((r_\beta,z_\beta)\big)}\zeta_{_{\beta,\Lambda}} d\nu \ge 1-\frac{d\gamma}{r_*},
\end{equation}
which implies the desired result for all small $\gamma$ such that $1-\frac{d\gamma}{r_*}>1/2$. It follows that $diam\big({supp}(\zeta_{\beta,\Lambda})\big)\le C/\log\frac{1}{\beta}$ provided $\beta$ is small enough. With this in hand, we can improve $\eqref{212}$ as follows
\begin{equation*}
\begin{split}
   I_1 &=\int\limits_{D\cap\{\sigma>\beta^\gamma\}}G(r_\beta,z_\beta,r',z')\zeta_{_{\beta,\Lambda}}(r',z') r'dr'dz'\\
       &\le \frac{(r_\beta)^{\frac{1}{2}}}{2\pi} \sinh^{-1}(\frac{1}{\beta^\gamma})\int\limits_{D\cap\{\sigma>\beta^\gamma\}}\zeta_{_{\beta,\Lambda}}(r',z')r'^{\frac{3}{2}}dr'dz'\\
       &\le \frac{r_\beta}{2\pi} \sinh^{-1}(\frac{1}{\beta^\gamma})+C.
\end{split}
\end{equation*}
We can now repeat the proof and sharpen $\eqref{216}$ as follows
\begin{equation*}
  \liminf_{\beta\to 0^+}\int\limits_{B_{2d\beta^\gamma}((r_\beta,z_\beta))}\zeta_{_{\beta,\Lambda}}d\nu \ge 1-\frac{\gamma}{2},
\end{equation*}
which implies the desired result and completes the proof.
\end{proof}

\begin{lemma}\label{le9}
There holds
\begin{equation}\label{217}
  \lim_{\beta\to0^+}\frac{\log \big[diam\big({supp}(\zeta_{\beta,\Lambda})\big)\big]}{\log{\beta}}=1.
\end{equation}
\end{lemma}

\begin{proof}
Since $0\le \zeta_{_{\beta,\Lambda}} \le \Lambda/\beta^2$ and $\int_D\zeta_{_{\beta,\Lambda}} d\nu=1$, we have $m_2\big(supp(\zeta_{_{\beta,\Lambda}})\big)\ge C\beta^2$. This implies that
\begin{equation}\label{218}
  diam\big(supp(\zeta_{\beta,\Lambda})\big)\ge C\beta.
\end{equation}
Combining $\eqref{218}$ with Lemma $\ref{diam1}$, it is easy to obtain $\eqref{217}$. The proof is completed.
\end{proof}

We can summarize what we have proved as the following proposition.
\begin{proposition}\label{le10}
 For any $\gamma\in(0,1)$, there holds
 \begin{equation*}
   diam\big(supp(\zeta_{\beta,\Lambda})\big) \le 4d \beta^\gamma
 \end{equation*}
 provided $\beta$ is small enough.
Moreover, one has
\begin{equation*}
\begin{split}
   \lim_{\beta\to0^+}\frac{\log \big[diam\big({supp}(\zeta_{_{\beta,\Lambda}})\big)\big]}{\log{\beta}}&=1, \\
   \lim_{\beta\to0^+}dist_{\mathcal{C}_{r_{_*}}}\big(supp(\zeta_{_{\beta,\Lambda}})\big)&=0.
\end{split}
\end{equation*}
\end{proposition}

The following lemma shows that  $\psi_{_{\beta,\Lambda}}$ has a prior upper bound when $W>1/(4\pi d)$.
\begin{lemma}\label{ad2}
  Suppose $W>1/(4\pi d)$, then for all sufficiently small $\beta$, we have
\begin{equation*}
  (\psi_{_{\beta,\Lambda}})_+ \le C_1 \log \Lambda +C_2,
\end{equation*}
where the positive constants $C_1$, $C_2$ are independent of $\beta$, $\Lambda$.
\end{lemma}

\begin{proof}
  For any $(r,z)\in supp\,(\zeta_{_{\beta,\Lambda}})$, by Lemmas \ref{le1}, \ref{le2} and Proposition \ref{le10}, we have
\begin{equation*}
  \begin{split}
      \psi_{\beta,\Lambda}(r,z)& =\mathcal{K}\zeta_{_{\beta,\Lambda}}(r,z)-\frac{Wr^2}{2}\log{\frac{1}{\beta}}-\mu_{_{\beta,\Lambda}} \\
       & \le \int_D G(r,z,r',z')\zeta_{_{\beta,\Lambda}}d\nu-\frac{Wr^2}{2}\log{\frac{1}{\beta}}-\mu_{_{\beta,\Lambda}} \\
       & \le \frac{A_\beta}{2\pi}\int_D\log{\big[(r-r')^2+(z-z')^2 \big]^{-\frac{1}{2}}}\zeta_{_{\beta,\Lambda}}(r', z')r'dr'dz'-\frac{W(A_\beta)^2}{2}\log{\frac{1}{\beta}}-\mu_{_{\beta,\Lambda}}+C_3 \\
       & \le \frac{(A_\beta)^2}{2\pi}\int_D\log{\big[(r-r')^2+(z-z')^2 \big]^{-\frac{1}{2}}}\zeta_{_{\beta,\Lambda}}(r', z')dr'dz'-\frac{W(A_\beta)^2}{2}\log{\frac{1}{\beta}}-\mu_{_{\beta,\Lambda}}+C_4 \\
       & \le \frac{(A_\beta)^2}{2\pi}\log\frac{\sqrt{\Lambda}}{\beta}\int_D\zeta_{_{\beta,\Lambda}}(r', z')dr'dz'-\frac{W(A_\beta)^2}{2}\log{\frac{1}{\beta}}-\mu_{_{\beta,\Lambda}}+C_5 \\
       & \le \Big(\frac{A_\beta}{2\pi}-\frac{W(A_\beta)^2}{2}\Big)\log{\frac{1}{\beta}}-\mu_{_{\beta,\Lambda}}+C_6\log \Lambda+C_6, \\
  \end{split}
\end{equation*}
where the positive number $C_6$ does not depend on $\beta$ and $\Lambda$. On the other hand, since $\lim_{\beta\to 0^+}A_\beta= r_*\in(0,d)$, one can deduce from Lemmas \ref{le4} and \ref{le5} that
\begin{equation*}
  \mu_{_{\beta,\Lambda}}\ge \Big(\frac{A_\beta}{2\pi}-\frac{W(A_\beta)^2}{2}\Big)\log{\frac{1}{\beta}}-C_7,
\end{equation*}
where the positive number $C_7$ is independent of $\beta$ and $\Lambda$.
Therefore we obtain
\begin{equation*}
  (\psi_{_{\beta,\Lambda}})_+ \le C_6\log \Lambda+C_8,
\end{equation*}
which clearly implies the desired result, and the proof is thus completed.
\end{proof}

Using Lemma \ref{ad2}, we can further deduce the following result.
\begin{proposition}\label{ad3}
   Suppose $W>1/(4\pi d)$, then there exists $\Lambda_0>1$ such that
  \begin{equation*}
    \zeta_{_{\beta, \Lambda_0}}=\frac{1}{r^2\beta^2}{(\psi_{_{\beta, \Lambda_0}})}_{+}+\frac{\alpha}{\beta}\chi_{_{\{\psi_{_{\beta, \Lambda_0}}>0\}}}
  \end{equation*}
 provided $\beta$ is sufficiently small.
\end{proposition}

\begin{proof}
  Notice that
\begin{equation*}
  (\psi_{_{\beta,\Lambda}})_+\ge (\Lambda-\alpha\beta)r^2\ \ \text{on}\ \  \{\zeta_{_{\beta,\Lambda}}=\frac{\Lambda}{\beta^2}\} ,
\end{equation*}
from which we deduce that
\begin{equation}\label{0-4}
  (\psi_{_{\beta,\Lambda}})_+\ge C_1\Lambda-C_2,
\end{equation}
for some positive numbers $C_1$, $C_2$ independent of $\beta$, $\Lambda$.
Combining \eqref{0-4} with Lemma \ref{ad2}, we conclude that
\begin{equation*}
  C_3\log\Lambda\ge C_1\Lambda-C_4,
\end{equation*}
which is impossible when $\Lambda$ is chosen to be sufficiently large, say $\Lambda \ge \Lambda_0$. Hence $$m_2(\{\zeta_{_{\beta,\Lambda_0}}=\frac{\Lambda}{\beta^2}\})=0,$$ when $\beta$ is small enough, the proof is completed.
\end{proof}

\begin{remark}\label{rem1}
For all sufficiently small $\beta$, the stream-function $\psi_{_{\beta,\Lambda_0}}$ satisfies
\begin{equation*}
  \mathcal{L}\psi_{_{\beta,\Lambda_0}}=f(r,\psi_{_{\beta,\Lambda_0}})\ \ \ \  in\  D,
\end{equation*}
where $f(r,t)=\frac{t_+}{r^2\beta^2}+\frac{\alpha}{\beta}\chi_{\{t>0\}}$.
\end{remark}

For the rest of this section, we shall abbreviate $(\zeta_{_{\beta,\Lambda_0}},\psi_{_{\beta,\Lambda_0}}, \mu_{_{\beta,\Lambda_0}})$ to $(\zeta_{_\beta}, \psi_{_\beta}, \mu_{_\beta})$. When the concentrated location is inside the domain, we may have the following sharp estimates.
\begin{lemma}\label{le11}
Suppose $W>1/(4\pi d)$, then as $\beta \to 0^+$,
\begin{equation}\label{219}
E_\beta(\zeta_{_\beta})=(\frac{r_*}{4\pi}-\frac{Wr_*^2}{2})\log{\frac{1}{\beta}}+O(1),
\end{equation}
\begin{equation}\label{220}
\mu_{_\beta}=(\frac{r_*}{2\pi}-\frac{Wr_*^2}{2})\log{\frac{1}{\beta}}+O(1).
\end{equation}
\end{lemma}

\begin{proof}
We first prove $\eqref{219}$. According to Proposition $\ref{le10}$, there holds
\begin{equation*}
  supp\,(\zeta_{_\beta})\subseteq B_{4d\beta^{\frac{1}{2}}}((A_\beta,0))
\end{equation*}
for all sufficiently small $\beta$. Hence, by Lemmas $\ref{le1}$ and $\ref{le2}$, we have
\begin{equation*}
  \int_D\zeta_{_\beta} \mathcal{K} \zeta_{_\beta} d\nu \le\frac{(B_\beta)^3}{2\pi}\iint_{D\times D}\log[(r-r')^2+(z-z')^2]^{-\frac{1}{2}}\zeta_{_\beta}(r,z)\zeta_{_\beta}(r',z')drdzdr'dz'+O(1).
\end{equation*}
By Lemma 4.2 of \cite{Tur83}, we have
\begin{equation*}
\iint_{D\times D}\log[(r-r')^2+(z-z')^2]^{-\frac{1}{2}}\zeta_{_\beta}(r,z)\zeta_{_\beta}(r',z')drdzdr'dz'\le \log\frac{1}{\beta}\{\int_D\zeta_{_\beta} drdz\}^2+O(1).
\end{equation*}
Notice that
\begin{equation*}
 \int_D\zeta_{_\beta} drdz\le \frac{1}{A_\beta}.
\end{equation*}
Therefore
\begin{equation*}
   \int_D\zeta_{_\beta} \mathcal{K} \zeta_{_\beta} d\nu\le \frac{(B_\beta)^3}{2\pi (A_\beta)^2}\log\frac{1}{\beta}+O(1)\le \frac{A_\beta}{2\pi}\log\frac{1}{\beta}+O(1),
\end{equation*}
from which we conclude that
\begin{equation*}
  E_\beta(\zeta_{_\beta})\le (\frac{A_\beta}{4\pi}-\frac{W(A_\beta)^2}{2})\log{\frac{1}{\beta}}+O(1)\le (\frac{r_*}{4\pi}-\frac{Wr_*^2}{2})\log{\frac{1}{\beta}}+O(1).
\end{equation*}
Combining this with Lemma $\ref{le4}$, we get $\eqref{219}$. Now $\eqref{220}$ clearly follows from Lemma \ref{ad1}, the proof is completed.
\end{proof}

Using Lemma $\ref{le11}$, we can sharpen the first estimate in Proposition $\ref{le10}$ .
\begin{lemma}\label{sharp}
 Suppose $W>1/(4\pi d)$, then there exists a constant $R_1>1$ independent of $\beta$ such that $diam(supp(\zeta_{_\beta}))\le R_1\beta$.
\end{lemma}

\begin{proof}
We follow the strategy of \cite{Tur83}. Let us use the same notation as in the proof of Lemma $\ref{le6}$. According to Proposition $\ref{le10}$, there holds
\begin{equation*}
  supp(\zeta_{_\beta})\subseteq B_{4d\beta^{\frac{1}{2}}}\big((A_\beta,0)\big)
\end{equation*}
for all sufficiently small $\beta$.
Let $R>1$ to be determined. By Lemma $\ref{le2}$, we have
\begin{equation}\label{221}
\begin{split}
   I_1 &=\int\limits_{D\cap\{\sigma>R\beta\}}G(r_\beta,z_\beta\,r',z')\zeta_{_\beta}(r',z')r'dr'dz'\\
       &\le \frac{(A_\beta)^2+O(\beta^{\frac{1}{2}})}{2\pi}\int\limits_{D\cap\{\sigma>R\beta\}}\log{\frac{1}{\sigma}}\zeta_{_\beta}(r',z') dr'dz'+C\\
       &\le \frac{(A_\beta)^2+O(\beta^{\frac{1}{2}})}{2\pi}\log{\frac{1}{R\beta}}\int\limits_{D\cap\{\sigma>R\beta\}}\zeta_{_\beta}(r',z') dr'dz'+C\\
       &\le \frac{A_\beta}{2\pi}\log{\frac{1}{R\beta}}\int\limits_{D\cap\{\sigma>R\beta\}}\zeta_{_\beta} d\nu+C,
\end{split}
\end{equation}
and
\begin{equation}\label{222}
\begin{split}
   I_2 &=\int\limits_{D\cap\{\sigma \le R\beta\}}G(r_\beta,z_\beta,r',z')\zeta_{_\beta}(r',z')r'dr'dz'\\
       &\le \frac{(A_\beta)^2+O(\beta^\frac{1}{2})}{2\pi}\int\limits_{D\cap\{\sigma \le R\beta\}}\log\frac{1}{[(r_\beta-r')^2+(z_\beta-z')^2]^{\frac{1}{2}}}\zeta_{_\beta}(r',z')dr'dz'+C\\
       &\le \frac{(A_\beta)^2+O(\beta^\frac{1}{2})}{2\pi}\log\frac{1}{\beta}\int\limits_{D\cap\{\sigma \le R\beta\}}\zeta_{_\beta} dr'dz'+C\\
       & \le \frac{A_\beta}{2\pi}\log\frac{1}{\beta}\int\limits_{D\cap\{\sigma \le R\beta\}}\zeta_{_\beta} d\nu+C.
\end{split}
\end{equation}
On the other hand, by Lemma $\ref{le11}$, one has
\begin{equation}\label{223}
  \mu_{_\beta}\ge(\frac{A_\beta}{2\pi}-\frac{W(A_\beta)^2}{2})\log\frac{1}{\beta}-C.
\end{equation}
Combining $\eqref{220}$, $\eqref{221}$ and $\eqref{222}$, we obtain
\begin{equation*}
  \frac{A_\beta}{2\pi}\log\frac{1}{\beta} \le \frac{A_\beta}{2\pi}\log{\frac{1}{R\beta}}\int\limits_{D\cap\{\sigma>R\beta\}}\zeta_{_\beta} d\nu+\frac{A_\beta}{2\pi}\log\frac{1}{\beta}\int\limits_{D\cap\{\sigma \le R\beta\}}\zeta_{_\beta} d\nu+C.
\end{equation*}
Hence
\begin{equation*}
  \int\limits_{D\cap\{\sigma \le R\beta\}}\zeta_{_\beta} d\nu\ge 1-\frac{C}{\log R}.
\end{equation*}
Taking $R$ large enough such that $C(\log R)^{-1}<1/2$, we obtain
\begin{equation*}
  \int\limits_{D\cap{B_{2dR\beta}((r_\beta,z_\beta))}}\zeta_{_\beta} d\nu\ge\int\limits_{D\cap\{\sigma \le R\beta\}}\zeta_{_\beta} d\nu> \frac{1}{2}.
\end{equation*}
Taking $R_1=4dR$, we clearly get the desired result.
\end{proof}

We now investigate the asymptotic shape of the optimal vortices when $W>1/(4\pi d)$.
We define the center of vorticity to be
\begin{equation}\label{224}
  X_\beta=\int x\zeta_\beta(x)d\textit{m}_2(x)
\end{equation}
Then by Proposition $\ref{le10}$ we know that $X_\beta \to (r_*,0)$ as $\beta \to 0^+$. Let $g_{_\beta} \in L^\infty(B_{R_1}(0))$ be defined by
\begin{equation}\label{225}
  g_{_\beta}(x)={\beta^2}\zeta_{_\beta}(X_\beta+\beta x)
\end{equation}
for fixed $R_1$ as in Lemma $\ref{sharp}$. We denote by $g^\bigtriangleup_{_\beta}$ the symmetric radially nonincreasing Lebesgue-rearrangement of $g_\beta$ centered on 0. It is clearly that $supp(g^\bigtriangleup_{_\beta})\subseteq B_{R_2\beta}(0)$ for some $R_2\ge R_1$.

The following result determines the asymptotic nature of $\zeta_{_\beta}$ in terms of its scaled version.

\begin{lemma}\label{le12}
Suppose $W>1/(4\pi d)$, then every accumulation point of the family $\{g_{_\beta}:\beta>0\}$ in the weak topology of $L^2$ must be a radially nonincreasing function.
\end{lemma}

\begin{proof}
Let us assume that $g_{_\beta} \to g$ and $g^\bigtriangleup_\beta \to h$ weakly in $L^2$ for some functions $g$ and $h$ as $\beta \to 0^+$. Firstly, by virtue of Riesz' rearrangement inequality, we have
\begin{equation*}
     \iint\limits_{B_{R_2}(0)\times B_{R_2}(0)}\log\frac{1}{|x-x'|}g_{_\beta}(x)g_{_\beta}(x')d(x,x')
      \le \iint\limits_{B_{R_2}(0)\times B_{R_2}(0)}\log\frac{1}{|x-x'|}g^\bigtriangleup_{_\beta}(x)g^\bigtriangleup_{_\beta}(x')d(x,x').
\end{equation*}
Hence
\begin{equation}\label{226}
  \iint\limits_{B_{R_2}(0)\times B_{R_2}(0)}\log\frac{1}{|x-x'|}g(x)g(x')d(x,x')\le \iint\limits_{B_{R_2}(0)\times B_{R_2}(0)}\log\frac{1}{|x-x'|}h(x)h(x')d(x,x').
\end{equation}
Let $\tilde{\zeta}_\beta$ be defined as
\[
\tilde{\zeta}_\beta(x)=\left\{
   \begin{array}{lll}
        \beta^2g^\bigtriangleup_{_\beta}(\beta^{-1}(x-X_\beta)) &    \text{if} & x\in B_{R_1\beta}(X_\beta), \\
         0                  &    \text{if} & x\in D\backslash B_{R_1\beta}(X_\beta).
    \end{array}
   \right.
\]
A direct calculation then yields as $\beta\to 0$,
\begin{equation*}
\begin{split}
     E_\beta(\zeta_{_\beta}) =&\frac{(A_\beta)^3}{4\pi}\iint\limits_{B_{R_2}(0)\times B_{R_2}(0)}\log\frac{1}{|x-x'|}g_\beta(x)g_\beta(x')d(x,x') \\
     & +\frac{(A_\beta)^3}{4\pi}\big{\{}\int_D \zeta_{_\beta} d\textit{m}_2\big{\}}^2-\frac{W(A_\beta)^3}{2}\log\frac{1}{\beta}\int_D\zeta_{_\beta} d\textit{m}_2+\mathcal{R}_\beta(\zeta_{_\beta})+o(1),
\end{split}
\end{equation*}
and
\begin{equation*}
\begin{split}
     E_\beta(\tilde{\zeta}_{_\beta}) =&\frac{(A_\beta)^3}{4\pi}\iint\limits_{B_{R_2}(0)\times B_{R_2}(0)}\log\frac{1}{|x-x'|}g^\bigtriangleup_{_\beta} (x)g^\bigtriangleup_{_\beta} (x')d(x,x') \\
     & +\frac{(A_\beta)^3}{4\pi}\big{\{}\int_D \zeta_{_\beta} d\textit{m}_2\big{\}}^2-\frac{W(A_\beta)^3}{2}\log\frac{1}{\beta}\int_D\zeta_{_\beta} d\textit{m}_2+\mathcal{R}_\beta(\tilde{\zeta}_{_\beta})+o(1),
\end{split}
\end{equation*}
where
\begin{equation*}
  \lim_{\beta\to 0^+}\mathcal{R}_\beta(\zeta_{_\beta})=\lim_{\beta\to 0^+}\mathcal{R}_\beta(\tilde{\zeta}_{_\beta})< \infty.
\end{equation*}
Recalling that $E_\beta(\tilde{\zeta}_{_\beta})\le E_\beta(\zeta_{_\beta})$, we conclude that
\begin{equation*}
  \iint\limits_{B_{R_2}(0)\times B_{R_2}(0)}\log\frac{1}{|x-x'|}g(x)g(x')d(x,x')\le \iint\limits_{B_{R_2}(0)\times B_{R_2}(0)}\log\frac{1}{|x-x'|}h(x)h(x')d(x,x'),
\end{equation*}
which together with $\eqref{226}$ yield to
\begin{equation*}
    \iint\limits_{B_{R_2}(0)\times B_{R_2}(0)}\log\frac{1}{|x-x'|}g(x)g(x')d(x,x')= \iint\limits_{B_{R_2}(0)\times B_{R_2}(0)}\log\frac{1}{|x-x'|}h(x)h(x')d(x,x').
\end{equation*}
By Lemma 3.2 of \cite{BG}, we know that there exists a translation $T$ in $\mathbb{R}^2$ such that $Tg=h$. Note that
\begin{equation*}
  \int_{B_{R_2}(0)}xg(x)d\textit{m}_2=\int_{B_{R_2}(0)}xh(x)d\textit{m}_2=0.
\end{equation*}
Thus $g=h$, the proof is completed.
\end{proof}

\begin{proof}[Proof of Theorem \ref{thm1}]
Let $(\psi_{_\beta}, \zeta_{_\beta}, \mu_{_\beta})=(\psi_{_{\beta, \Lambda_0}}, \zeta_{_{\beta, \Lambda_0}}, \mu_{_{\beta, \Lambda_0}})$ as obtained in Lemma \ref{le3}, where $\Lambda_0$ is the large number given in Proposion \ref{ad3}.
  We now set $\xi_{_\beta}=\frac{1}{\beta}(\psi_{_\beta})_+$.  Then (i) and (ii) follow from Lemmas \ref{le3}, \ref{ad2} and Proposion \ref{ad3}. And (iii) follows from Proposition \ref{le10} and Lemmas \ref{le11}, \ref{sharp}. To prove (iv), we refer to \cite{B1}. (v) follows from Lemma \ref{le12}. It remains to prove that $(\psi_{_\beta}, \xi_{_\beta}, \zeta_{_\beta})$ satisfies \eqref{1-22} and \eqref{1-23}. Note that \eqref{1-22} is clear, we prove \eqref{1-23} here. For any $\varphi\in C_0^\infty(D)$, we have
\begin{equation*}
\begin{split}
     \int_D \Big[\zeta_{_\beta} \partial\big(\psi_{_\beta},\varphi\big)-\frac{1}{2r^2}\partial\big((\xi_{_\beta})^2,\varphi\big)\Big]drdz & =\int_D \Big[\zeta_{_\beta} \partial\big(\psi_{_\beta},\varphi \big)-\frac{(\psi_{_\beta})_+}{r^2\beta^2}\partial\big(\psi_{_\beta},\varphi\big)\Big]drdz \\
     & =\int_D\Big[\big(\zeta_{_\beta}-\frac{(\psi_{_\beta})_+}{r^2\beta^2}\big)\partial\big(\psi_{_\beta},\varphi\big)  \Big]drdz \\
     & =\int_D \frac{\alpha}{\beta}\chi_{\{\psi_{_\beta}>0\}}\partial(\psi_{_\beta},\varphi)drdz \\
     & =0,
\end{split}
\end{equation*}
where the last equality follows from Lemma \ref{burton}. The proof is completed.
\end{proof}

\section{Vortex Rings Outside a Ball}

In this section, we investigate vortex rings outside a ball. Although the idea is the same as in Section 3, the approach here is somewhat different.

Let $D=\{(r,z)\in \Pi ~|~ r^2+z^2>d^2\}$ for some fixed $d>0$.

Let $q(r,z)=r^2/(r^2+z^2)^{\frac{3}{2}}$, it is easy to check that $\mathcal{L}q=0$. For fixed $\alpha\ge0$, $\beta>0$ and $W>0$, we consider the energy as follows
\[E_\beta(\zeta)=\frac{1}{2}\int_D{\zeta \mathcal{K}\zeta}d\nu-\frac{{W}}{2}\log{\frac{1}{\beta}}\int_{D}r^2\zeta d\nu+\frac{W}{2}\log\frac{1}{\beta}\int_{D}\frac{r^2d^3}{(r^2+z^2)^{\frac{3}{2}}}\zeta d\nu-\frac{\beta^2}{2}\int_D r^2(\zeta-\frac{\alpha}{\beta})_{+}^2d\nu.\]

We introduce the function $\Gamma_2$ as follows
\begin{equation*}
  \Gamma_2(t)=\frac{t}{2\pi}-Wt^2+\frac{Wd^3}{t},\ t\in(0,+\infty).
\end{equation*}
Let $r_*\in[d,+\infty)$ such that $\Gamma_2(r_*)=\max_{t\in[d,+\infty)}\Gamma_2(t)$. It is easy to check that $r_*$ is unique. Moreover, $r_*=d$ if $W\ge1/(6\pi d)$; $r_*>d$ if $W<1/(6\pi d)$.
Let $D_1=\{(r,z)\in D~|~0<r<r_*+1, |z|<d+1\}$. We adopt the class of admissible functions $\mathcal{A}_{\beta,\Lambda}(D_1)$ as follows
\begin{equation*}
\mathcal{A}_{\beta,\Lambda}(D_1)=\{\zeta\in L^\infty(D)~|~ supp(\zeta)\subseteq D_1,\  0\le \zeta \le \frac{\Lambda}{\beta^2}~ \text{a.e.}, \int_{D}\zeta d\nu \le1\},
\end{equation*}
where the positive number $\Lambda>\max\{\alpha\beta,1\}$.

For any $\zeta\in \mathcal{A}_{\beta,\Lambda}(D_1)$, one can check that $\mathcal{K}\zeta \in W^{2,p}_{\text{loc}}(D)\cap C_{\text{loc}}^{1,\gamma}(\bar{D})$ for any $p>1$, $0<\gamma<1$.

\begin{lemma}\label{le13}
For every prescribed  $\alpha\ge0$, $1>\beta>0$, $W>0$ and $\Lambda>\max\{\alpha \beta,1\}$, there exists $\zeta_{_{\beta,\Lambda}} \in \mathcal{A}_{\beta,\Lambda}$ such that
\begin{equation}\label{400}
 E_\beta(\zeta_{_{\beta,\Lambda}})= \max_{\tilde{\zeta} \in \mathcal{A}_{\beta,\Lambda}}E_\beta(\tilde{\zeta})<+\infty.
\end{equation}
Moreover, there exists a Lagrange multiplier $\mu_{_{\beta,\Lambda}}\in \mathbb{R}$ such that
\begin{equation}\label{401}
\zeta_{_{\beta,\Lambda}}=\Big(\frac{\psi_{_{\beta,\Lambda}}}{r^2\beta^2}+\frac{\alpha}{\beta}\Big){\chi}_{_{D_1\cap\{0<\psi_{_{\beta,\Lambda}}<(\Lambda-\alpha\beta)r^2\}}}+\frac{\Lambda}{\beta^2}\chi_{_{D_1\cap\{\psi_{_{\beta,\Lambda}} \ge (\Lambda-\alpha\beta)r^2\}}} \ \ a.e.\  \text{in}\  D,
\end{equation}
where
\begin{equation*}
 \psi_{_{\beta,\Lambda}}=\mathcal{K}\zeta_{_{\beta,\Lambda}}-\frac{Wr^2}{2}\log{\frac{1}{\beta}}+\frac{Wr^2d^3}{2(r^2+z^2)^{\frac{3}{2}}}\log\frac{1}{\beta}-\mu_{_{\beta,\Lambda}}. 
\end{equation*}
Furthermore, whenever $E(\zeta_{_{\beta,\Lambda}})>0$ and $\mu_{_{\beta,\Lambda}}>0$ there holds $\int_D\zeta_{_{\beta,\Lambda}}d\nu=1$ and $\zeta_{_{\beta,\Lambda}}$ has compact support in $D$.
\end{lemma}

\begin{proof}
  By Lemma $\ref{le2}$, we deduce that $K(r,z,r',z')\in L^1(D_1\times D_1)$, which further implies that the functional $\int_D \zeta \mathcal{K} \zeta d\nu$ is weakly-star continuous in $L^\infty(D_1\times D_1)$. With this in hand, it is not hard to obtain the existence of $\zeta_\beta$. Since the rest of the proof is similar to that in Lemma $\ref{le3}$, we omit the details here and complete our proof
  of  Lemma $\ref{le13}$.
\end{proof}

 Arguing as in the proof of Lemma $\ref{le4}$, we can obtain the lower bound of the energy.
\begin{lemma}\label{le14}
  For any  $a\in(d,r_*+1)$, there exists $C>0$ such that for all $\beta$ sufficiently small, we have
\begin{equation*}
  E_\beta(\zeta_{_{\beta,\Lambda}})\ge (\frac{a}{4\pi}-\frac{Wa^2}{2}+\frac{Wd^3}{2a})\log{\frac{1}{\beta}}-C,
\end{equation*}
where the positive number $C$ depends on $a$, but not on $\beta$, $\Lambda$.
\end{lemma}

We now turn to study the multiplier $\mu_{_{\beta,\Lambda}}$. The following lemma is a counterpart of Lemma \ref{ad1}.
\begin{lemma}\label{ad4}
For all sufficiently small $\beta$, there holds
\begin{equation*}
\begin{split}
   \mu_{_{\beta,\Lambda}} =&2E_\beta(\zeta_{_{\beta,\Lambda}})+\frac{W}{2}\log\frac{1}{\beta}\int_D\zeta_{_{\beta,\Lambda}} r^2 d\nu-\frac{Wd^3}{2}\log\frac{1}{\beta}\int_{D}q(r,z)\zeta_{_{\beta,\Lambda}}d\nu  \\
     & -\int_D{\alpha \beta r^2(\zeta_{_{\beta,\Lambda}}-\frac{\alpha}{\beta})_+}d\nu-\frac{\Lambda}{\beta^2}\int_D U_{\beta,\Lambda}d\nu,
\end{split}
\end{equation*}
where $U_{\beta,\Lambda}:=\big[\psi_{_{\beta,\Lambda}}-(\Lambda-\alpha\beta)r^2\big]_+$.
Moreover, one has
\begin{equation*}
  \mu_{_{\beta,\Lambda}}=2E_\beta(\zeta_{_{\beta,\Lambda}})+\frac{W}{2}\log\frac{1}{\beta}\int_D\zeta_{_{\beta,\Lambda}} r^2 d\nu-\frac{Wd^3}{2}\log\frac{1}{\beta}\int_{D}q(r,z)\zeta d\nu+O(1),
\end{equation*}
 where the bounded quantity $O(1)$ does not depend on $\beta$ and $\Lambda$.
\end{lemma}

\begin{proof}
We first claim that $\mu_{_{\beta,\Lambda}}>-1/2$ when $\beta$ is small enough. Indeed, if $\mu_{_{\beta,\Lambda}}<0$, then
\begin{equation*}
  \{(r,z)\in D_1~|~\frac{Wr^2}{2}\log\frac{1}{\beta}<|\mu_{_{\beta,\Lambda}}|\}\subseteq \{\zeta_{_{\beta,\Lambda}}>\frac{\alpha}{\beta}\}
\end{equation*}
Since $|\{\zeta_{_{\beta,\Lambda}}>\frac{\alpha}{\beta}\}|\le \beta^2/\Lambda$, we obtain $|\mu_{_{\beta,\Lambda}}|\le \frac{W\beta^2}{2\Lambda}\log\frac{1}{\beta}$, which clearly implies the claim. Set
\begin{equation*}
  \tilde{U}_{{\beta,\Lambda}}=\big[\psi_{_{\beta,\Lambda}}-(\Lambda-\alpha\beta)r^2-1\big]_+.
\end{equation*}
Now we can argue as in the proof of \eqref{0-3}, with ${U}_{{\beta,\Lambda}}$ replaced by $\tilde{U}_{{\beta,\Lambda}}$, to obtain
\begin{equation*}
  \frac{\Lambda}{\beta^2}\int_D U_{\beta,\Lambda}d\nu \le \frac{\Lambda}{\beta^2}\int_D \tilde{U}_{\beta,\Lambda}d\nu+1\le C,
\end{equation*}
where the positive number $C$ does not depend on $\beta$ and $\Lambda$. The rest of proof is the same as before, we omit it here. The proof is completed.
\end{proof}

Combining Lemmas \ref{le14} and \ref{ad4}, we immediately obtain:
\begin{lemma}\label{le15}
  For any  $a\in(d,r_*+1)$, there exists $C>0$ such that for all $\beta$ sufficiently small, we have
 \begin{equation}\label{227}
    \mu_{_{\beta,\Lambda}}\ge(\frac{a}{2\pi}-{Wa^2}+\frac{Wd^3}{a})\log{\frac{1}{\beta}}+\frac{W}{2}\log{\frac{1}{\beta}}\int_D\zeta_{\beta} r^2d\nu-\frac{Wd^3}{2}\log\frac{1}{\beta}\int_{D}q(r,z)\zeta d\nu-C,
 \end{equation}
 where the positive number $C$ depends on $a$, but not on $\beta$, $\Lambda$.
\end{lemma}

We now study the asymptotic behaviour of the vortex core. Define
\[A_{\beta}=\inf\{r_\beta~|~(r_\beta,z_\beta)\in supp(\zeta_{_{\beta,\Lambda}})\ \text{for some}\  z_{\beta} \in \mathbb{R}\},\]
\[B_{\beta}=\sup\{r_\beta~|~(r_\beta,z_\beta)\in supp(\zeta_{_{\beta,\Lambda}})\ \text{for some}\  z_{\beta} \in \mathbb{R} \}.\]
We may assume that $(A_\beta,Z_{\beta1})$, $(B_\beta,Z_{\beta2})\in supp(\zeta_{_{\beta,\Lambda}})$ for some $Z_{\beta1}$, $Z_{\beta2} \in \mathbb{R}$.

\begin{lemma}\label{le16}
  $\lim_{\beta \to 0^+}A_\beta=r_*$.
\end{lemma}

\begin{proof}
  Let $\sigma$ be defined by $\eqref{201}$ and $\gamma\in(0,1)$. Let $(r_\beta,z_\beta)\in {supp}\,(\zeta_{_{\beta,\Lambda}})$, arguing as in the proof of Lemma $\ref{le6}$, we may obtain that for some positive number $C_0$ not depending on $\beta$, $\Lambda$ and $\gamma$,
\begin{equation*}
 \mathcal{K}\zeta_{_{\beta,\Lambda}}(r_\beta,z_\beta)\le \frac{r_\beta}{2\pi}\log\frac{1}{\beta}\int\limits_{B_{C_0\beta^\gamma}\big((r_\beta,z_\beta)\big)}\zeta_{_{\beta,\Lambda}} d\nu+C\sinh^{-1}(\frac{1}{\beta^\gamma})+C.
\end{equation*}
On the other hand, for any $a\in(1,r_*+1)$, we have
\begin{equation*}
\begin{split}
   \mathcal{K}\zeta_{_{\beta,\Lambda}}&(r_\beta,z_\beta)-\frac{W(r_\beta)^2}{2}\log\frac{1}{\beta}+\frac{Wd^3}{2}\log \frac{1}{\beta}\,g(r_\beta,z_\beta) \\
    & \ge (\frac{a}{2\pi}-{Wa^2}+\frac{Wd^3}{a})\log{\frac{1}{\beta}}+\frac{W}{2}\log{\frac{1}{\beta}}\int_D\zeta_{_{\beta,\Lambda}} r^2d\nu-\frac{Wd^3}{2}\log\frac{1}{\beta}\int_{D}q(r,z)\zeta_{_{\beta,\Lambda}} d\nu-C.
\end{split}
\end{equation*}
Hence
\begin{equation*}
\begin{split}
   \frac{r_\beta}{2\pi}\log\frac{1}{\beta}+C\sinh^{-1}&(\frac{1}{\beta^\gamma})-\frac{W(r_\beta)^2}{2}\log\frac{1}{\beta}+\frac{Wd^3}{2}\log \frac{1}{\beta}\,q(r_\beta,z_\beta)\ge  \\
     & \Gamma_2(a)\log{\frac{1}{\beta}}+\frac{W}{2}\log{\frac{1}{\beta}}\int_D\zeta_{_{\beta,\Lambda}} r^2d\nu-\frac{Wd^3}{2}\log\frac{1}{\beta}\int_{D}q(r,z)\zeta_{_{\beta,\Lambda}} d\nu-C.
\end{split}
\end{equation*}
Dividing both sides of the above inequality by $\log\frac{1}{\beta}$, we get
\begin{equation}\label{228}
\begin{split}
   \Gamma_2(a)\le ~&\Gamma_2(r_\beta)+\frac{W}{2}\big{\{}(r_\beta)^2-\int_D\zeta_{_{\beta,\Lambda}} r^2d\nu \big{\}}  +\frac{Wd^3}{2}\big{\{}q(r_\beta,z_\beta)+\int_Dq(r,z)\zeta_{_{\beta,\Lambda}}d\nu -\frac{2}{r_\beta}\big{\}} \\
     &  +\frac{C\sinh^{-1}(\frac{1}{\beta^{\gamma}})}{\log\frac{1}{\beta}}+\frac{C}{\log\frac{1}{\beta}}.
\end{split}
\end{equation}
Note that $q(r,z)\ge q(r,0)=1/r$. Taking $(r_\beta,z_\beta)=(A_\beta,Z_{\beta1})$, we obtain
\begin{equation*}
  \Gamma_2(a)\le \Gamma_2(A_\beta)+\frac{C\sinh^{-1}(\frac{1}{\beta^{\gamma}})}{\log\frac{1}{\beta}}+\frac{C}{\log\frac{1}{\beta}}.
\end{equation*}
Letting $\beta \to 0^+$, it follows that
\begin{equation*}
  \Gamma_2(a)\le \liminf_{\beta\to 0^+}\Gamma_2(A_\beta)+C\gamma.
\end{equation*}
Letting $a\to r_*$ and $\gamma \to 0^+$, we deduce that
\begin{equation*}
  \Gamma_2(r_*)\le \liminf_{\beta\to 0^+}\Gamma_2(A_\beta),
\end{equation*}
which implies the desired result. The proof is completed.
\end{proof}

Arguing as before, one can easily obtain
\begin{lemma}\label{le17}
There holds
\begin{equation*}
\begin{split}
   \lim_{\beta\to 0^+}\int_D\zeta_{_{\beta,\Lambda}} r^2d\nu &=r_*^2,  \\
   \lim_{\beta\to 0^+}\int_D q(r,z)\zeta_{_{\beta,\Lambda}} d\nu &=\frac{1}{r_*},\\
   \lim_{\beta\to 0^+}B_\beta&=r_*.
\end{split}
\end{equation*}
\end{lemma}

We may further prove that the vortex core will shrink to the circle $\mathcal{C}_{r_*}$.
\begin{lemma}\label{le18}
$\lim_{\beta \to 0^+}dist_{\mathcal{C}_{r_{_*}}}\big(supp(\zeta_{_{\beta,\Lambda}})\big)=0.$
\end{lemma}

\begin{proof}
It suffices to prove that for any $(r_\beta, z_\beta)\in supp(\zeta_{_{\beta,\Lambda}})$, we have $\lim_{\beta\to 0^+}z_\beta=0$. Indeed, it follows from $\eqref{228}$  that for any $\gamma \in (0,1)$, there holds
\begin{equation*}
  \liminf_{\beta\to 0^+}q(r_\beta,z_\beta)\ge \frac{1}{r_*}-\frac{C\gamma}{W}.
\end{equation*}
Since $\lim_{\beta \to 0^+}r_\beta=r_*$, we must have $\lim_{\beta\to 0^+}z_\beta=0$. The proof is completed.
\end{proof}

With these results in hand, one can easily obtain the following results.
\begin{proposition}\label{le19}
For any $\gamma\in (0,1)$, there holds
\begin{equation*}
  \begin{split}
     diam\big(supp(\zeta_{_{\beta,\Lambda}})\big)  \le C\beta^\gamma,& \\
     dist\big(supp(\zeta_{_{\beta,\Lambda}}),\partial D_1\big)  > 0, &
  \end{split}
\end{equation*}
provided $\beta$ is small enough, where $C>0$ is independent of $\beta$, $\Lambda$ and $\gamma$. Moreover,
\begin{equation*}
  \lim_{\beta\to 0^+}\frac{\log\big[diam\big(supp(\zeta_{_{\beta,\Lambda}})\big)\big]}{\log{\beta}}=1.
\end{equation*}
\end{proposition}

Now, we further study the profile of $\zeta_{_{\beta,\Lambda}}$.
\begin{lemma}\label{lex1}
  For all sufficiently small $\beta$, there holds
\begin{equation}\label{x1}
  \{\zeta_{_{\beta,\Lambda}}>0\}=\{(r,z)\in D\,|\,\psi_{_{\beta,\Lambda}}>0\}.
\end{equation}
\end{lemma}

\begin{proof}
By virtue of \eqref{401}, we first know that $ \{\zeta_{_{\beta,\Lambda}}>0\}=\{(r,z)\in D_1\,|\,\psi_{_{\beta,\Lambda}}>0\}$.
Note that
\begin{equation*}
  \begin{split}
      \mathcal{L}\psi_{_{\beta,\Lambda}}&=0\ \ \text{in}\  D\backslash  supp(\zeta_{_{\beta,\Lambda}}), \\
      \psi_{_{\beta,\Lambda}}&\le 0\ \ \text{on}\  \partial D \cup \partial \big(supp(\zeta_{_{\beta,\Lambda}})\big), \\
     \psi_{_{\beta,\Lambda}} &\le 0 \ \ \text{at}\  \infty .
  \end{split}
\end{equation*}
By the maximum principle, we conclude that $\psi_{_{\beta,\Lambda}}\le 0$ in $D\backslash  \{\zeta_{_{\beta,\Lambda}}>0\}$. Hence,
\begin{equation*}
  \{\zeta_{_{\beta,\Lambda}}>0\}=\{(r,z)\in D\,|\,\psi_{_{\beta,\Lambda}}>0\}.
\end{equation*}
The proof is completed.
\end{proof}

Arguing as in the proof of Proposition \ref{ad3}, we get

\begin{proposition}\label{ad5}
   Suppose $W<1/(6\pi d)$, then there exists $\Lambda_0>1$ such that
  \begin{equation*}
    \zeta_{_{\beta, \Lambda_0}}=\frac{1}{r^2\beta^2}{(\psi_{_{\beta, \Lambda_0}})}_{+}+\frac{\alpha}{\beta}\chi_{_{\{\psi_{\beta,\Lambda_0}>0\}}}
  \end{equation*}
 provided $\beta$ is sufficiently small.
\end{proposition}

\begin{remark}\label{rem2}
For all sufficiently small $\beta$, the stream-function $\psi_{_{\beta,\Lambda_0}}$ satisfies
\begin{equation*}
  \mathcal{L}\psi_{_{\beta,\Lambda_0}}=f(r,\psi_{_{\beta,\Lambda_0}})\ \ in\  D,\ \ \ \psi=-\mu_\beta\ \ on\ \partial D,
\end{equation*}
where $f(r,t)=\frac{t_+}{r^2\beta^2}+\frac{\alpha}{\beta}\chi_{\{t>0\}}$.
\end{remark}

For the rest of this section, we shall abbreviate $(\zeta_{_{\beta, \Lambda_0}},\psi_{_{\beta, \Lambda_0}}, \mu_{_{\beta, \Lambda_0}})$ to $(\zeta_{_\beta}, \psi_{_\beta}, \mu_{_\beta})$. As before, we have the following sharp estimates when the vortices concentrate in the interior of the domain. Moreover, let $g_\beta$ be defined as in Section 3, then the asymptotic nature of $\zeta_{_\beta}$ can also be obtained.
\begin{proposition}\label{le20}
Suppose $W<1/(6\pi d)$, then as $\beta \to 0^+$,
\begin{equation*}
\begin{split}
   E_\beta(\zeta_{_\beta}) & =(\frac{r_*}{4\pi}-\frac{Wr_*^2}{2}+\frac{Wd^3}{2r_*})\log{\frac{1}{\beta}}+O(1),\\
   \mu_{_\beta}  & =(\frac{r_*}{2\pi}-\frac{Wr_*^2}{2}+\frac{Wd^3}{2r_*})\log{\frac{1}{\beta}}+O(1).
\end{split}
\end{equation*}
Also, there exists a constant $R_1>1$ independent of $\beta$ such that
 $$diam\big(supp(\zeta_{_\beta})\big)\le R_1\beta,$$
 provided $\beta$ is sufficiently small. Moreover, every accumulation point of the family $\{g_{_\beta}:\beta>0\}$ in the weak topology of $L^2$ must be a radially nonincreasing function.
\end{proposition}

 We are now ready to give proof of Theorem \ref{thm2}.
\begin{proof}[Proof of Theorem \ref{thm2}]
Having made all the above preparations, the proof is similar to Theorem \ref{thm1}; consequently, we omit them here.
\end{proof}

\section{Vortex Rings In The Whole Space}
In this section, we study vortex rings in the whole space. Because the approach used is similar to the previous two sections, we list only some key steps without providing detailed proof.

Let $D=\Pi$. For fixed $\alpha\ge0$, $\beta>0$ and $W>0$, we consider the energy functional as follows
\[E_\beta(\zeta)=\frac{1}{2}\int_D{\zeta \mathcal{K}\zeta}d\nu-\frac{{W}}{2}\log{\frac{1}{\beta}}\int_{D}r^2\zeta d\nu-\frac{\beta^2}{2}\int_D r^2(\zeta-\frac{\alpha}{\beta})_{+}^2d\nu.\]
Let $D_2=\{(r,z)\in D~|~0<r<1/(2\pi W)\}$. We adopt the class of admissible functions $\mathcal{A}_{\beta,\Lambda}$ as follows
\begin{equation*}
\mathcal{A}_{\beta,\Lambda}=\{\zeta\in L^\infty(D)~|~ 0\le \zeta \le \frac{\Lambda}{\beta^2}~ \text{a.e.}, \int_{D}\zeta d\nu \le1,~ supp(\zeta)\subseteq D_2 \},
\end{equation*}
where $\Lambda>\max\{\alpha \beta,1\}$ is a fixed positive number.

For any $\zeta\in \mathcal{A}_{\beta,\Lambda}$, one can check that $\mathcal{K}\zeta \in W^{2,p}_{\text{loc}}(D)\cap C_{\text{loc}}^{1,\gamma}(\bar{D})$ for any $p>1$, $0<\gamma<1$. Recall that ${\zeta}^*$ denotes the Steiner symmetrization of $\zeta$ with respect to the line $z=0$ in $D$.

\begin{lemma}\label{le5-1}
For every prescribed  $\alpha\ge0$, $1>\beta>0$, $W>0$ and $\Lambda>\max\{\alpha \beta,1\}$, there exists $\zeta=\zeta_{_{\beta,\Lambda}} \in \mathcal{A}_{\beta,\Lambda}$ such that
\begin{equation}\label{501}
 E_\beta(\zeta_{_{\beta,\Lambda}})= \max_{\tilde{\zeta} \in \mathcal{A}_{\beta,\Lambda}}E_\beta(\tilde{\zeta})<+\infty.
\end{equation}
Moreover, there exists a Lagrange multiplier $\mu_{_{\beta,\Lambda}} \ge 0$ such that
\begin{equation}\label{502}
\zeta_{_{\beta,\Lambda}}=\Big(\frac{\psi_{_{\beta,\Lambda}}}{r^2\beta^2}+\frac{\alpha}{\beta}\Big){\chi}_{_{D_2\cap\{0<\psi_{_{\beta,\Lambda}}<(\Lambda-\alpha\beta)r^2\}}}+\frac{\Lambda}{\beta^2}\chi_{_{D_2\cap\{\psi_{\beta,\Lambda} \ge (\Lambda-\alpha\beta)r^2\}}} \ \ a.e.\  \text{in}\  D,
\end{equation}
where
\begin{equation}\label{503}
 \psi_{_{\beta,\Lambda}}=\mathcal{K}\zeta_{_{\beta,\Lambda}}-\frac{Wr^2}{2}\log{\frac{1}{\beta}}-\mu_{_{\beta,\Lambda}}.
\end{equation}
Also, $\zeta_{_{\beta,\Lambda}}=(\zeta_{_{\beta,\Lambda}})^*$. Furthermore, whenever $E(\zeta_{_{\beta,\Lambda}})>0$ and $\mu_{_{\beta,\Lambda}}>0$ there holds $\int_D\zeta_{_{\beta,\Lambda}} d\nu=1$ and $\zeta_{_{\beta,\Lambda}}$ has compact support in $D$.
\end{lemma}

The following result is a counterpart of Proposition \ref{ad5} in Section 4.
\begin{proposition}\label{le5-2}
   There exists $\Lambda_0>1$ such that
  \begin{equation*}
    \zeta_{_{\beta, \Lambda_0}}=\frac{1}{r^2\beta^2}{(\psi_{_{\beta, \Lambda_0}})}_{+}+\frac{\alpha}{\beta}\chi_{_{\{\psi_{\beta,\Lambda_0}>0\}}}
  \end{equation*}
 provided $\beta$ is sufficiently small.
\end{proposition}

As before, we now abbreviate $(\zeta_{_{\beta, \Lambda_0}},\psi_{_{\beta, \Lambda_0}}, \mu_{_{\beta, \Lambda_0}})$ to $(\zeta_{_{\beta}}, \psi_{_{\beta}}, \mu_{_{\beta}})$. Let $g_{_\beta}$ be defined as in Section 3. Set $r_*=1/(4\pi W)$.
\begin{proposition}\label{le20}
As $\beta \to 0^+$, we have
\begin{equation*}
\begin{split}
   E_\beta(\zeta_{_\beta}) & =(\frac{r_*}{4\pi}-\frac{Wr_*^2}{2})\log{\frac{1}{\beta}}+O(1),\\
   \mu_{_\beta} & =(\frac{r_*}{2\pi}-\frac{Wr_*^2}{2})\log{\frac{1}{\beta}}+O(1).
\end{split}
\end{equation*}
Also, there exists a constant $R_1>1$ independent of $\beta$ such that
 $$diam\big(supp(\zeta_{_\beta})\big)\le R_1\beta,$$
provided $\beta$ is small enough. Moreover,
$$\lim_{\beta \to 0^+}dist_{\mathcal{C}_{r_{_*}}}(supp(\zeta_{_\beta}))=0.$$
In addition, every accumulation point of the family $\{g_{_\beta}:\beta>0\}$ in the weak topology of $L^2$ must be a radially nonincreasing function.
\end{proposition}

\begin{proof}[Proof of Theorem \ref{thm3}]
Having made the above preparations, the proof is similar to Theorem \ref{thm1}; consequently, we omit them here.
\end{proof}

{\bf Acknowledgments.}
{ This work was supported by NNSF of China (No.11771469) and Chinese Academy of Sciences (
  No.QYZDJ-SSW-SYS021).
}

\phantom{s}
 \thispagestyle{empty}

\end{document}